\documentclass[a4paper,12pt]{amsart}
\title[Stability conditions and birational geometry]{{\bf Stability conditions and birational geometry 
of projective surfaces}}
\date{}
\author{Yukinobu Toda}


\usepackage{amscd}
\usepackage{amsmath}
\usepackage{amssymb}
\usepackage{amsthm}
\usepackage{float}
\usepackage[dvips]{graphicx}

\usepackage[all,ps,dvips]{xy}

\usepackage{array}
\usepackage{amscd}
\usepackage[all]{xy}

\DeclareFontFamily{U}{rsfs}{%
\skewchar\font127}
\DeclareFontShape{U}{rsfs}{m}{n}{%
<-6>rsfs5<6-8.5>rsfs7<8.5->rsfs10}{}
\DeclareSymbolFont{rsfs}{U}{rsfs}{m}{n}
\DeclareSymbolFontAlphabet
{\mathrsfs}{rsfs}
\DeclareRobustCommand*\rsfs{%
\@fontswitch\relax\mathrsfs}

\theoremstyle{plain}
\newtheorem{thm}{Theorem}[section]
\newtheorem{prop}[thm]{Proposition}
\newtheorem{lem}[thm]{Lemma}

\newtheorem{defi}[thm]{Definition}
\newtheorem{rmk}[thm]{Remark}
\newtheorem{cor}[thm]{Corollary}

\newtheorem{step}{Step}
\newtheorem{sstep}{Step}

\newtheorem{case}{Case}
\newtheorem{ccase}{Case}

\newtheorem{prop-defi}[thm]{Proposition-Definition}
\newtheorem{thm-defi}[thm]{Theorem-Definition}
\newtheorem{lem-defi}[thm]{Lemma-Definition}

\newtheorem{question}[thm]{Question}

\newtheorem{exam}[thm]{Example}

\newdimen\argwidth
\def\db[#1\db]{
 \setbox0=\hbox{$#1$}\argwidth=\wd0
 \setbox0=\hbox{$\left[\box0\right]$}
  \advance\argwidth by -\wd0
 \left[\kern.3\argwidth\box0 \kern.3\argwidth\right]}

\newcommand{\aA}{\mathcal{A}}

\newcommand{\cC}{\mathcal{C}}
\newcommand{\dD}{\mathcal{D}}
\newcommand{\eE}{\mathcal{E}}
\newcommand{\fF}{\mathcal{F}}

\newcommand{\hH}{\mathcal{H}}

\newcommand{\mM}{\mathcal{M}}

\newcommand{\oO}{\mathcal{O}}
\newcommand{\pP}{\mathcal{P}}
\newcommand{\qQ}{\mathcal{Q}}

\newcommand{\sS}{\mathcal{S}}
\newcommand{\tT}{\mathcal{T}}
\newcommand{\uU}{\mathcal{U}}

\newcommand{\Hom}{\mathop{\rm Hom}\nolimits}

\newcommand{\dR}{\mathbf{R}}
\newcommand{\dL}{\mathbf{L}}

\newcommand{\id}{\textrm{id}}

\newcommand{\ch}{\mathop{\rm ch}\nolimits}

\newcommand{\Ext}{\mathop{\rm Ext}\nolimits}

\newcommand{\Coh}{\mathop{\rm Coh}\nolimits}

\newcommand{\cneq}{\mathrel{\raise.095ex\hbox{:}\mkern-4.2mu=}}
\newcommand{\eqcn}{\mathrel{=\mkern-4.5mu\raise.095ex\hbox{:}}}

\newcommand{\Cok}{\mathop{\rm Cok}\nolimits}

\newcommand{\Ex}{\mathop{\rm Ex}\nolimits}

\newcommand{\Stab}{\mathop{\rm Stab}\nolimits}
\newcommand{\PPer}{\mathop{\rm Per}\nolimits}

\newcommand{\Imm}{\mathop{\rm Im}\nolimits}

\newcommand{\Ker}{\mathop{\rm Ker}\nolimits}

\begin{document}
\maketitle

\begin{abstract}
We show that the minimal model program on any smooth
projective 
surface is realized as a variation of 
the moduli spaces of 
Bridgeland stable objects in the derived
category of coherent sheaves. 
\end{abstract}

\setcounter{tocdepth}{1}
\tableofcontents

\section{Introduction}
\subsection{Motivation}
This paper is a continuation of the previous paper~\cite{Todext}, 
in which the following question 
on the relationship between minimal model program (MMP)
and Bridgeland stability conditions~\cite{Brs1}
was addressed
(cf.~\cite[Question~1.1]{Todext}):
\begin{question}\label{intro:quest}
Let $X$ be a smooth projective variety and 
consider its MMP
\begin{align*}
X=X_1 \dashrightarrow X_2 \dashrightarrow \cdots \dashrightarrow X_N.
\end{align*}
Then is each $X_i$ a
moduli space of Bridgeland (semi)stable objects in the 
derived category of coherent sheaves on $X$, and MMP 
is interpreted as wall-crossing under a variation of 
Bridgeland stability conditions?
\end{question} 
The main result of~\cite{Todext}
was to answer the above question for the first 
step of MMP, i.e. extremal contraction, when 
$\dim X \le 3$. 
The purpose of this paper is to 
give a complete answer to the 
above question for further steps of MMP
when $\dim X=2$. 

\subsection{Bridgeland stability}
For a smooth projective variety $X$, 
Bridgeland~\cite{Brs1}
introduced the notion of stability conditions on 
$D^b \Coh(X)$, 
which provides a mathematical framework of Douglas's 
$\Pi$-stability~\cite{Dou2}. 
In~\cite{Brs1}, Bridgeland showed that
the set of stability conditions 
\begin{align}\label{intro:stab}
\Stab(X)
\end{align}
forms a complex manifold, 
and studied it 
when $X$ is a K3 surface or an abelian surface~\cite{Brs2}.
Since then there have been several studies
on the space (\ref{intro:stab}), or 
the associated moduli spaces of semistable objects
in the derived category, 
 when $X$ is a K3 surface or an abelian 
surface (cf.~\cite{HMS}, \cite{AB}, 
\cite{Tst3}, \cite{MYY}, \cite{MYY2}, \cite{YY}). 

On the other hand, there are few papers in the 
literature
in which the space (\ref{intro:stab})
is studied 
for an arbitrary projective surface $X$.
If $X$ is non-minimal, 
the birational geometry of $X$ is interesting, 
and 
we expect that it has a deep connection with the
space of stability conditions (\ref{intro:stab}). 
This idea is motivated by Bridgeland's work~\cite{Br1} on the 
construction of three dimensional flops as moduli spaces
of objects in the derived category. 
This result is not yet possible to realize in terms of 
Bridgeland stability conditions, since 
constructing them on projective 
3-folds turned out to be a
very difficult problem (cf.~\cite{BMT}).
However in the surface case, we have the examples 
of stability conditions constructed
by Arcara-Bertram~\cite{AB}.
Given the above background, we shall 
establish a 
rigorous statement
connecting two dimensional MMP and the space of Bridgeland
stability conditions (\ref{intro:stab}).

\subsection{Main result}
Our main result is formulated in the space 
$\Stab(X)_{\mathbb{R}}$, 
defined to be the 
`real part' of the space (\ref{intro:stab}). 
This is the space which fits into the Cartesian square
(cf.~Section~\ref{sec:append})
\begin{align}\label{Cart}
\xymatrix{ \ar@{}[rd]|{\square}
\Stab(X)_{\mathbb{R}} \ar[r] \ar[d]_{\Pi_{\mathbb{R}}} & 
\Stab(X) \ar[d]^{\Pi} \\
\mathrm{NS}(X)_{\mathbb{R}}
 \ar[r]^{- \int_X e^{-i\ast}} & N(X)_{\mathbb{C}}^{\vee}.
}
\end{align}
Recall that the 
ample cone $A(X) \subset \mathrm{NS}(X)_{\mathbb{R}}$
plays an important role in birational geometry. 
We will see that there is an open subset
\begin{align}\label{intro:U}
U(X) \subset \Stab(X)_{\mathbb{R}}
\end{align}
which is homeomorphic to 
$A(X)$ under the map $\Pi_{\mathbb{R}}$
of the diagram (\ref{Cart}). 
The subset (\ref{intro:U}) coincides with 
the set of $\sigma \in \Stab(X)_{\mathbb{R}}$ in which 
all the objects $\oO_x$ for $x\in X$ are stable. 
The closure $\overline{U}(X)$ is the analogue of the nef cone
of $X$, and expected to contain information of the birational 
geometry of $X$. 

Our purpose is to construct 
an open subset such as (\ref{intro:U})
associated to each birational morphism 
$f \colon X \to Y$,
and investigate how they are 
related under the change of $(f, Y)$. 
Here we fix the notation: 
for a Bridgeland stability 
condition $\sigma=(Z, \aA)$, 
we denote by $\mM^{\sigma}([\oO_x])$ the algebraic 
space which parameterizes $Z$-stable objects 
$E \in \aA$ with phase one and 
$\ch(E)=\ch(\oO_x)$
for $x\in X$ (cf.~\cite{Inaba}).
The following is the main theorem in this paper:  
\begin{thm}{\bf (Proposition~\ref{prop:main1}, 
Proposition~\ref{prop:main2})}
\label{thm:intro}
Let $X$ be a smooth projective complex surface. 
Then for
any smooth projective surface $Y$ and a 
birational morphism $f\colon X \to Y$,
there is a
connected open subset
\begin{align*}
U(Y) \subset \Stab(X)_{\mathbb{R}}
\end{align*}
satisfying the following conditions:
\begin{itemize}
\item If $f$ factors through a
blow-up at a point $Y' \to Y$, 
we have 
\begin{align}\label{intro:wall}
\overline{U}(Y) \cap \overline{U}(Y') \neq \emptyset
\end{align}
which is of real codimension one in $\Stab(X)_{\mathbb{R}}$. 
\item For any $\sigma \in U(Y)$, 
$\mM^{\sigma}([\oO_x])$ is isomorphic to $Y$. 
\end{itemize} 
\end{thm}
The above result shows that 
the space $\Stab(X)_{\mathbb{R}}$ is 
a fundamental object, beyond the ample 
cone $A(X)$,  in the study of birational 
geometry of $X$. Indeed, the geometry of any 
birational morphism 
$f \colon X \to Y$ is captured from the space
$\Stab(X)_{\mathbb{R}}$. 
Here is a simple example of Theorem~\ref{thm:intro}:
\begin{exam}\label{exam:P2}
Let $f \colon X \to \mathbb{P}^2$ be the blow-up 
at a point $p\in \mathbb{P}^2$. 
Let $H$ be the pull-back of a line in $\mathbb{P}^2$
and $C$ the exceptional curve of $f$. 
Then $[H]$ and $[C]$ span $\mathrm{NS}(X)_{\mathbb{R}}$. 
The subsets 
$U(X)$ and $U(\mathbb{P}^2)$ in $\Stab(X)_{\mathbb{R}}$ are homeomorphic to 
their images in $\Pi_{\mathbb{R}}$, 
and they are given by
\begin{align*}
&\Pi_{\mathbb{R}}(U(X))=\{ x[H]+y[C] : x>0, -x<y<0\} \\
&\Pi_{\mathbb{R}}(U(\mathbb{P}^2))
=\{x[H] +y[C] ; x>0, 0<y<x \}. 
\end{align*}
\end{exam}

The following is the obvious corollary of Theorem~\ref{thm:intro}:
\begin{cor}
Let $X$ be a smooth projective complex surface
and 
\begin{align*}
X=X_1 \to X_2 \to \cdots \to X_N
\end{align*}
a MMP, i.e. contractions of $(-1)$-curves. Then 
there is a continuous one parameter family of Bridgeland 
stability conditions $\{ \sigma_t\}_{t\in (0, 1)}$
on $D^b \Coh(X)$
and real numbers 
\begin{align*}
0=t_0<t_1< t_2< \cdots <t_{N}=1
\end{align*}
such that 
$X_i$ is isomorphic to $\mM^{\sigma_t}([\oO_x])$
for $t\in (t_{i-1}, t_i)$. 
\end{cor}
The result of the above corollary completely 
answers Question~\ref{intro:quest}
for surfaces: any MMP of a smooth projective 
surface $X$ is realized as wall-crossing of 
Bridgeland moduli spaces of stable objects
in $D^b \Coh(X)$. 
The real numbers $t_i$ correspond to walls in 
this case. 

\subsection{Technical ingredients}
As we mentioned before, the space (\ref{intro:stab})
has not been studied for an arbitrary projective
surface $X$. Although it was studied for a 
K3 surface or an abelian surface in~\cite{Brs2}, there 
are several technical arguments in~\cite{Brs2}
which are not applied directly to an
arbitrary projective surface $X$. 
It seems that these technical issues have
prevented us to study the space (\ref{intro:stab})
beyond the case of K3 surfaces or abelian surfaces. 

One of the technical issues is to prove the support 
property of the stability conditions. This property is 
required in order to make the topology of the space (\ref{intro:stab})
desirable. It has now turned out that proving the 
support property is not an easy problem in general, 
and closely related to the Bogomolov-Gieseker (BG) type inequality 
of semistable objects in the derived category. 
In the case of K3 surfaces or abelian surfaces, proving 
the BG type inequality is easier: 
this follows from the Serre duality and the Riemann-Roch 
theorem. However this is not the case for an arbitrary 
projective surface, and we need to find 
a general argument
proving such an inequality. 
In the previous paper~\cite{Todext}, we 
established such a BG type inequality 
for semistable objects on an arbitrary 
projective surface, 
and 
proved the support property for some 
stability conditions in 
$\overline{U}(X)$. 
We use this result to show the support property 
for stability conditions contained in other 
subsets $\overline{U}(Y)$. 

Another issue is that the analysis of the boundary 
of $U(X)$ in the case of K3 surfaces in~\cite{Brs1} is not applied for 
an arbitrary projective surface $X$. 
In the former case, if we cross the codimension one boundary of $U(X)$, 
then the resulting stability condition is obtained by applying 
some autoequivalence of the derived category. 
In the latter case, this is not the case in general. 
Indeed we will see that, after crossing the boundary 
of $U(X)$ corresponding to a $(-1)$-curve contraction, 
then the resulting stability condition is 
not described by an autoequivalence but by
a certain tilting of the t-structure which appears at the 
boundary. We will describe the resulting tilting 
explicitly, and investigate the 
wall-crossing behavior of the open subsets $U(Y)$ 
in Theorem~\ref{thm:intro} in detail.

\subsection{Relation to existing works}
There are some recent works in which 
the relationship between Bridgeland stability 
conditions and MMP is discussed
(cf.~\cite{ABCH}, \cite{BaMa2}, \cite{Tst}, \cite{Todext}).
The works~\cite{ABCH}, \cite{BaMa2} treat the cases of 
$\mathbb{P}^2$ and K3 surfaces respectively. 
Also the works~\cite{Tst}, \cite{Todext} treat
the cases of local flops, contraction of a $(-1)$-curve, respectively. 
The result in this paper generalizes the result of~\cite{Todext}, 
and completely answer~\cite[Question~1.1]{Todext} for an 
arbitrary projective surface. 

The examples of 
Bridgeland stability conditions on arbitrary projective surfaces 
are given in~\cite{AB}. 
In the works~\cite{MYY}, \cite{MYY2}, \cite{YY}, 
\cite{Maci}, \cite{MaciMe}, 
the structure of walls and wall-crossing phenomena 
with respect to these stability conditions are studied. 
Our construction of $U(Y)$ provides other examples of
Bridgeland stability conditions on arbitrary non-minimal 
surfaces. 
It would be interesting to study the moduli 
spaces of semistable objects in $U(Y)$ with 
arbitrary numerical classes, and 
investigate their behavior under crossing the 
intersection of the closures (\ref{intro:wall}).

\subsection{Plan of the paper}
In Section~\ref{sec:append}, we give some background on 
Bridgeland stability conditions, especially on projective
surfaces. 
In Section~\ref{sec:construction}, we construct some 
t-structures on relevant triangulated categories. 
In Section~\ref{sec:proof}, we give a proof of Theorem~\ref{thm:intro}. 
In Section~\ref{sec:tech}, we prove some technical results
which are stated in previous sections.

\subsection{Acknowledgement}
This work is supported by World Premier 
International Research Center Initiative
(WPI initiative), MEXT, Japan. This work is also supported by Grant-in Aid
for Scientific Research grant (22684002), 
and partly (S-19104002),
from the Ministry of Education, Culture,
Sports, Science and Technology, Japan.

\subsection{Notation and convention}
In this paper, all the varieties are defined over 
$\mathbb{C}$. 
For a triangulated 
category $\dD$ and 
a set of objects $\sS \subset \dD$, 
we denote by $\langle \sS \rangle$ 
the smallest extension closed subcategory of 
$\dD$ which contains objects in $\sS$. 
The category $\langle S \rangle$ is called 
the extension closure of $S$. 
For the heart
of a bounded t-structure 
$\aA \subset \dD$, 
we denote by $\hH^{i}_{\aA}(\ast)$ the $i$-th cohomology 
functor with respect to the t-structure with heart $\aA$. 
If $\sS$ is contained in $\aA$, 
the right orthogonal complement of $\sS$ in $\aA$
is defined by 
\begin{align*}
\sS^{\perp} \cneq \{ E\in \aA : 
\Hom(\sS, E)=0\}. 
\end{align*}

\section{Background}\label{sec:append}
In this section, we briefly recall 
Bridgeland stability conditions, 
 and prepare
some results which will be needed in the later sections.  
\subsection{Bridgeland stability conditions}
Let $X$ be a smooth projective variety and $N(X)$ the 
numerical Grothendieck group of $X$. 
This is the quotient of the usual Grothendieck 
group $K(X)$ by the subgroup of $E\in K(X)$
with $\chi(E, F)=0$ for any $F\in K(X)$, 
where $\chi(E, F)$ is the Euler pairing
\begin{align*}
\chi(E, F) \cneq \sum_{i\in \mathbb{Z}}
(-1)^i \dim \Ext^i(E, F). 
\end{align*}
\begin{defi}\label{lem:pair2} \emph{(\cite{Brs1})}
A stability condition on $X$ is a pair 
\begin{align}\label{pair2}
(Z, \aA), \quad \aA \subset D^b \Coh(X)
\end{align}
where $Z \colon N(X) \to \mathbb{C}$
is a group homomorphism and $\aA$ is the heart of a 
bounded t-structure, 
such that the following conditions hold: 
\begin{itemize}
\item For any non-zero $E\in \aA$, we have 
\begin{align}\label{pro:1}
Z(E) \in \{ r\exp(i\pi \phi) : r>0, \phi \in (0, 1] \}.
\end{align}
\item (Harder-Narasimhan property)
For any $E\in \aA$, there is a filtration in $\aA$
\begin{align*}
0=E_0 \subset E_1 \subset \cdots \subset E_N
\end{align*}
such that each subquotient $F_i=E_i/E_{i-1}$ is 
$Z$-semistable with 
$\arg Z(F_i)> \arg Z(F_{i+1})$. 
\end{itemize}
\end{defi}
Here an object $E\in \aA$ is $Z$-\textit{(semi)stable} 
if for any subobject $0\neq F \subsetneq E$
we have 
\begin{align*}
\arg Z(F) <(\le) \arg Z(E). 
\end{align*}
The group homomorphism $Z$
is called a \textit{central charge}. 
The central charges we use in this paper 
are of the form
\begin{align}\label{Zomega}
Z_{\omega}(E)= -\int_{X} e^{-i\omega} \ch(E)
\end{align}
for $\omega \in \mathrm{NS}(X)_{\mathbb{R}}$. 
If $\dim X=2$, we have
\begin{align}\label{def:Z}
Z_{\omega}(E)=
-\ch_2(E) + \frac{\omega^2}{2} \ch_0(E)
+ i \ch_1(E) \cdot \omega. 
\end{align}
We fix a norm $\lVert \ast \rVert$
on the finite dimensional vector 
space $N(X)_{\mathbb{R}}$. 
We need to put the following technical condition 
on the stability conditions: 
\begin{defi}\label{def:sprop}
A stability condition (\ref{pair2})
satisfies the support property if there 
is a constant $K>0$ such 
that for any 
non-zero $Z$-semistable object $E\in \aA$, we have  
\begin{align*}
\frac{\lVert E \rVert}{\lvert Z(E) \rvert} < K. 
\end{align*}
\end{defi}
The set $\Stab(X)$
is defined to be the set of stability conditions on 
$D^b \Coh(X)$ satisfying the support property. 
The following is the main result of~\cite{Brs1}. 
(Also see~\cite{K-S}.)
\begin{thm}\emph{(\cite{Brs1})}\label{thm:Brmain}
There is a natural topology 
on $\Stab(X)$ such that the forgetting map 
\begin{align*}
\Pi \colon 
\Stab(X) \to N(X)_{\mathbb{C}}^{\vee}
\end{align*}
sending $(Z, \aA)$ to $Z$ is a local homeomorphism. 
\end{thm}
We are interested in the set of 
stability conditions whose central charges are of 
the form (\ref{Zomega}). 
So we restrict our attention to 
the space $\Stab(X)_{\mathbb{R}}$ defined as follows: 
\begin{defi}
We define $\Stab(X)_{\mathbb{R}}$ 
to be the Cartesian square
\begin{align}\label{Car2}
\xymatrix{ \ar@{}[rd]|{\square}
\Stab(X)_{\mathbb{R}} \ar[r] \ar[d]_{\Pi_{\mathbb{R}}} & \Stab(X) 
\ar[d]^{\Pi} \\
\mathrm{NS}(X)_{\mathbb{R}}
 \ar[r]^{- \int_X e^{-i\ast}} & N(X)_{\mathbb{C}}^{\vee}.
}
\end{align}
Here the bottom map takes $\omega \in \mathrm{NS}(X)_{\mathbb{R}}$
to the central charge $Z_{\omega}$
given by (\ref{Zomega}). 
\end{defi}

\subsection{Gluing t-structures}\label{subsec:glue}
We use the following \textit{gluing t-structure} method 
in order to produce several t-structures. 
Let 
\begin{align*}
\cC \stackrel{i}{\to} \dD \stackrel{j}{\to} \eE
\end{align*}
be an \textit{exact triple} of triangulated categories. 
Namely $\cC$, $\dD$ and $\eE$ are triangulated categories, 
$i$, $j$ are exact functors with 
$j\circ i=0$. Both of $i$ and $j$ have the left and the right adjoint
functors, 
which satisfy some axioms. For the detail, see~\cite[IV.~Ex.~2]{GM}. 

Let 
\begin{align*}
(\cC^{\le 0}, \cC^{\ge 0}), \quad 
(\eE^{\le 0}, \eE^{\ge 0})
\end{align*}
be bounded t-structures on $\cC$ and $\eE$ respectively. 
Then they induce the bounded
t-structure on $\dD$ 
whose heart is given by 
\begin{align*}
\{ E \in \dD : 
j(E) \in \eE^{0}, \ 
\Hom(i(\cC^{<0}), E)= \Hom(E, i(\cC^{>0}))=0 \}. 
\end{align*}
Here $\eE^{0} \cneq \eE^{\le 0} \cap \eE^{\ge 0}$ is the heart
on $\eE$. 
For the detail, see~\cite[n.~1.4]{BBD}, \cite[IV.~Ex.~4]{GM}.

\subsection{Perverse t-structure}\label{subsec:perverse}
Let $X$ and $Y$ be smooth projective 
surfaces, and $f$ a birational morphism 
\begin{align*}
f \colon X\to Y. 
\end{align*}
We recall the construction of the perverse
t-structure associated to the above data, following~\cite{Br1}, \cite{MVB}. 

It is well-known that the derived pull-back
\begin{align*}
\dL f^{\ast} \colon D^b \Coh(Y) \to D^b \Coh(X)
\end{align*}
is fully-faithful. The functor $\dL f^{\ast}$ 
has the right adjoint $\dR f_{\ast}$ and the left 
adjoint $\dR f_{!}$, 
\begin{align*}
\dR f_{\ast}, \dR f_{!} \colon D^b \Coh(X) \to D^b \Coh(Y)
\end{align*}
where $\dR f_{!}$ is given by 
\begin{align*}
\dR f_{!} E= \dR f_{\ast}(E \otimes \omega_X) \otimes \omega_Y^{-1}. 
\end{align*}
We define the triangulated subcategories 
$\cC_{X/Y}$, $\dD_{X/Y}$ 
in $D^b \Coh(X)$ to be
\begin{align*}
\cC_{X/Y} &\cneq \{ E \in D^b \Coh(X) : 
\dR f_{!}E \cong 0\} \\
\dD_{X/Y} &\cneq \{ E \in D^b \Coh(X) :
\dR f_{\ast}E \cong 0\}. 
\end{align*}
They are related by $\cC_{X/Y} \otimes \omega _X = \dD_{X/Y}$. 
Here we only use the latter category 
$\dD_{X/Y}$.
The category $\cC_{X/Y}$ will be treated in the next section. 
 
We have the sequences of exact functors
\begin{align*}
\dD_{X/Y} \to D^b \Coh(X) \stackrel{\dR f_{\ast}}{\to} D^b \Coh(Y)
\end{align*}
where the left functor is the natural inclusion. 
The above sequence determines an exact triple, and 
the standard t-structure on $D^b \Coh(X)$ induces a 
t-structure 
\begin{align*}
(\dD_{X/Y}^{\le 0}, \dD_{X/Y}^{\ge 0})
\end{align*}
on $\dD_{X/Y}$ (cf.~\cite[Lemma~3.1]{Br1}).
By gluing the standard t-structure on $D^b \Coh(Y)$
and the shifted
t-structure $(\dD_{X/Y}^{\le -1}, \dD_{X/Y}^{\ge -1})$, we 
have the heart of the perverse t-structure (cf.~\cite{Br1}, \cite{MVB})
\begin{align*}
\PPer(X/Y) \subset D^b \Coh(X). 
\end{align*}
The perverse heart $\PPer(X/Y)$
is known to be 
equivalent to the module category of a 
certain sheaf of non-commutative 
coherent $\oO_Y$-algebras (cf.~\cite{MVB}). In particular, 
it is a noetherian abelian category. 
Also if $f=\id_X \colon X \to X$, the category $\PPer(X/X)$
coincides with $\Coh(X)$.

\subsection{Tilting of $\PPer(X/Y)$}\label{tilt:Y}
Let us take 
\begin{align*}
\omega \in \mathrm{NS}(Y)_{\mathbb{Q}}
\end{align*}
such that $\omega$ is a $\mathbb{Q}$-ample class. 
We have the following slope function, 
\begin{align*}
\mu_{f^{\ast}\omega} \colon \PPer(X/Y) \setminus \{0\} 
\to \mathbb{Q} \cup \{ \infty \}
\end{align*}
by setting $\mu_{f^{\ast}\omega}(E)=\infty$
if $\ch_0(E)=0$, and 
\begin{align*}
\mu_{f^{\ast}\omega}(E)= 
\frac{\ch_1(E) \cdot f^{\ast}\omega}{\ch_0(E)}
\end{align*}
if $\ch_0(E) \neq 0$. 
The above slope function determines a 
weak stability condition on $\PPer(X/Y)$, 
which satisfies the Harder-Narasimhan property
(cf.~\cite[Lemma~3.6]{Todext}).

We define 
the pair of subcategories $(\tT_{f^{\ast}\omega}, \fF_{f^{\ast}\omega})$
in $\PPer(X/Y)$
to be
\begin{align*}
\tT_{f^{\ast}\omega} &\cneq \langle E : E \mbox{ is }\mu_{f^{\ast}\omega}
\mbox{-semistable with } \mu_{f^{\ast}\omega}(E)>0 \rangle \\
\fF_{f^{\ast}\omega} &\cneq \langle E : E \mbox{ is }\mu_{f^{\ast}\omega}
\mbox{-semistable with } \mu_{f^{\ast}\omega}(E) \le 0 \rangle. 
\end{align*}
The above pair is a torsion pair~\cite{HRS} in $\Coh(X)$. 
The associated tilting is 
\begin{align}\label{A:tilt}
\aA_{f^{\ast}\omega} \cneq \langle 
\fF_{f^{\ast}\omega}[1], \tT_{f^{\ast}\omega} \rangle. 
\end{align}
By a general theory of tilting, the 
category $\aA_{f^{\ast}\omega}$ is the heart of 
a bounded t-structure on $D^b \Coh(X)$. 
In particular, it is an abelian category. 
Later we will need the following 
property on the above category. 
\begin{lem}\label{lem:pstab}
We have the embedding
\begin{align*}
\dL f^{\ast} \aA_{\omega} \subset \aA_{f^{\ast}\omega}. 
\end{align*}
\end{lem}
\begin{proof}
It is enough to show 
the following statements: 
\begin{itemize}
\item For any $M \in \Coh(Y)$, we have 
$\dL f^{\ast}M \in \PPer(X/Y)$. 
\item If $M$ is a torsion free $\mu_{\omega}$-semistable 
sheaf on $Y$, then $\dL f^{\ast}M \in \PPer(X/Y)$
is $\mu_{f^{\ast}\omega}$-semistable. 
\end{itemize}
We first show the first statement. 
By the projection formula, we have 
\begin{align*}
\dR f_{\ast} \dL f^{\ast}M \cong M. 
\end{align*}
Also we have 
\begin{align*}
\Hom(\dL f^{\ast}M, \dD_{X/Y}^{\ge 0}) \cong 0
\end{align*}
by adjunction. Let us take $F \in \dD_{X/Y}^{\le -2}$. We have 
\begin{align*}
\Hom(F, \dL f^{\ast}M) & \cong \Hom(\dR f_{!}F, M) \\
& \cong 0
\end{align*}
since $\dR f_{!}F \in \Coh^{\le -1}(Y)$. Therefore 
$\dL f^{\ast}M \in \PPer(X/Y)$ follows
by the definition of the gluing. 

As for the second statement, let us take an exact sequence in $\PPer(X/Y)$
\begin{align*}
0 \to F \to \dL f^{\ast}M \to G \to 0
\end{align*}
such that $F$ and $G$ are non-zero. We need to show that 
\begin{align}\label{need:show}
\mu_{f^{\ast}\omega}(F) \le \mu_{f^{\ast}\omega}(G). 
\end{align}
Applying $\dR f_{\ast}$, we obtain the exact 
sequence in $\Coh(Y)$
\begin{align*}
0 \to \dR f_{\ast}F \to M \to \dR f_{\ast}G \to 0. 
\end{align*}
If both of $\dR f_{\ast}(F)$ and $\dR f_{\ast}(G)$ are non-zero, 
the inequality (\ref{need:show}) holds 
by the $\mu_{\omega}$-stability of $E$
and noting $\mu_{f^{\ast}\omega}(\dL f^{\ast}(\ast))= \mu_{\omega}(\ast)$
for non-zero $\ast$. If $\dR f_{\ast}G=0$, then 
$\mu_{f^{\ast}\omega}(G)=\infty$ and (\ref{need:show}) holds. 
Suppose that $\dR f_{\ast}F=0$. 
Then $F \in \dD_{X/Y} \cap \Coh(X)[1]$, hence 
$\dR f_{!}F \in D^{\le 0}(\Coh(Y))$ 
and its zero-th cohomology is a zero dimensional sheaf. 
By adjunction and the torsion freeness 
of $M$, this implies 
\begin{align*}
\Hom(F, \dL f^{\ast}M) &\cong \Hom(\dR f_{!}F, M) \\
&\cong 0
\end{align*}
which is a contradiction. 
\end{proof}

\subsection{Bridgeland stability conditions on projective surfaces}\label{subsec:surface}
Let $f \colon X \to Y$ be a birational 
morphism between smooth projective surfaces, and 
$\omega \in \mathrm{NS}(Y)_{\mathbb{Q}}$
is ample. 
We consider the pair 
\begin{align*}
\sigma_{f^{\ast}\omega} \cneq 
(Z_{f^{\ast}\omega}, \aA_{f^{\ast}\omega}) 
\end{align*}
where 
$Z_{f^{\ast}\omega} \colon N(X) \to \mathbb{C}$ is the central charge 
defined by (\ref{def:Z}), and 
$\aA_{f^{\ast}\omega}$ is the heart of a bounded t-structure 
on $D^b \Coh(X)$
constructed in 
the previous subsection. 
We have the following proposition. 
\begin{prop}\label{prop:stab}
Suppose that 
$f$ satisfies one of the following conditions: 
\begin{itemize}
\item $f=\emph{\id}_X \colon X \to X$. 
\item $f$ contracts a single $(-1)$-curve $C$
on $X$ to a point in $Y$. 
\end{itemize}
Then we have 
\begin{align*}
\sigma_{f^{\ast}\omega} 
\in \Stab(X)_{\mathbb{R}}. 
\end{align*}
In particular, $\sigma_{f^{\ast}\omega}$ satisfies the support property. 
\end{prop}
\begin{proof}
If $f=\id_X$, the result
of~\cite{AB} shows that $\sigma_{f^{\ast}\omega}$
is a stability condition on $D^b \Coh(X)$. 
If $f$ contracts a $(-1)$-curve $C$ on $X$, the result 
of~\cite[Lemma~3.12]{Todext} shows that $\sigma_{f^{\ast}\omega}$ is a 
stability condition on $D^b \Coh(X)$. 
The support property of $\sigma_{f^{\ast}\omega}$ 
is proven in~\cite[Proposition~3.13]{Todext}
when $f$ contracts a $(-1)$-curve. 
When $f=\id_X$, the proof for the support property 
follows from the same (even easier)
argument of~\cite[Proposition~3.13]{Todext}.
\end{proof}
If $f=\id_X$, the stability condition 
$\sigma_{\omega}$ satisfies the following property: 
\begin{lem}\label{Ox:stable}
Let $\omega \in \mathrm{NS}(X)_{\mathbb{Q}}$ be ample 
and $f=\emph{\id}_X$. 

(i) For any $x\in X$, the object $\oO_x$ is a simple 
object in $\aA_{\omega}$. In particular, it
is $Z_{\omega}$-stable. 

(ii) For any object $E\in \aA_{\omega}$
with $\ch(E)=\ch(\oO_x)$, we have $E \cong \oO_x$
for some $x\in X$. 
\end{lem}
\begin{proof}
The result of 
(i) is essentially proved in~\cite[Lemma~6.3]{Brs2}. 
The result of (ii) is obvious from the construction of 
$\aA_{\omega}$. 
\end{proof}

The ample cone $A(X)$ is defined to be 
\begin{align*}
A(X) &\cneq \{ \omega \in \mathrm{NS}(X)_{\mathbb{R}} : \omega \mbox{ is }
\mathbb{R} \mbox{-ample} \}. 
\end{align*}
We define its 
partial compactification
$\overline{A}(X) \subset \mathrm{NS}(X)_{\mathbb{R}}$
to be 
\begin{align*}
\overline{A}(X) \cneq 
\bigcup_{f \colon X \to Y}
f^{\ast} A(Y). 
\end{align*}
In the above union, $f$ is either $\id_X \colon X \to X$ 
or contracts a single $(-1)$-curve on $X$
to a point in $Y$. 
Below, we sometimes write an element 
of $\overline{A}(X)$ as $\omega$ for a nef
divisor $\omega$ on $X$, omitting 
$f^{\ast}$ in the notation. 
We have the embedding 
\begin{align}\label{emb}
\overline{A}(X) \subset \mathrm{NS}(X)_{\mathbb{R}}. 
\end{align}
The following proposition shows 
the existence of stability conditions 
for irrational $\omega$: 
\begin{prop}\label{prop:lift}
The embedding (\ref{emb}) lifts to a continuous map  
\begin{align}\label{map:lift}
\sigma \colon \overline{A}(X) \to \Stab(X)_{\mathbb{R}}
\end{align}
which takes any rational point
 $\omega \in \overline{A}(X)$
to the stability condition
$\sigma_{\omega}$ in Proposition~\ref{prop:stab}. 
\end{prop}
\begin{proof}
The proof will be given in Subsection~\ref{pf:lift}. 
\end{proof}

\begin{rmk}
For $\omega \in A(X)$, it is possible to construct 
the heart $\aA_{\omega}$ similarly to (\ref{A:tilt}), 
even if $\omega$ is irrational. 
However the Harder-Narasimhan property 
for the pair $(Z_{\omega}, \aA_{\omega})$ is not 
obvious.
In the proof of Proposition~\ref{prop:lift},
we will also show that 
any object $\oO_x$ for $x\in X$ is 
$\sigma(\omega)$-stable, 
even when $\omega$ is irrational.  
Combined with 
Lemma~\ref{lem:Brlem},  
it shows that 
the pair $(Z_{\omega}, \aA_{\omega})$ indeed
satisfies the Harder-Narashiman property
for an irrational $\omega$.  
(Also see the argument of~\cite[Section~11]{Brs2}.)
\end{rmk}
We set $U(X) \subset \Stab(X)_{\mathbb{R}}$ to be
\begin{align*}
U(X) \cneq \sigma(A(X))
\end{align*}
Note that $U(X)$ is a connected open subset of $\Stab(X)_{\mathbb{R}}$,
which is homeomorphic to $A(X)$
under the forgetting map 
$\Stab(X)_{\mathbb{R}} \to \mathrm{NS}(X)_{\mathbb{R}}$. 
It satisfies the property of Theorem~\ref{thm:intro}
for $f=\id_X \colon X \to X$. 
Our purpose in the following sections is 
to construct a similar open subset 
associated to any birational morphism $f\colon X \to Y$. 

\section{Construction of t-structures}\label{sec:construction}
In what follows, 
$X$ and $Y$ are smooth projective surfaces and
\begin{align*}
f \colon X \to Y
\end{align*}
a birational morphism. 
In this section, we construct some t-structures
on $\cC_{X/Y}$ and $D^b \Coh(X)$,
which will be 
needed in the proof of Theorem~\ref{thm:intro}.

\subsection{t-structure on $\cC_{X/Y}$}
Let $\cC_{X/Y}$ be the triangulated
subcategory of $D^b \Coh(X)$
defined in Subsection~\ref{subsec:perverse}. 
The purpose here is to
 construct 
the heart of a bounded t-structure 
\begin{align*}
\cC_{X/Y}^{0} \subset \cC_{X/Y}
\end{align*}
satisfying the following conditions: 
there are objects $S_1, \cdots, S_N \in \cC_{X/Y}^{0}$
satisfying
\begin{align}\label{desire:S}
\cC_{X/Y}^{0} = \langle S_1, \cdots, S_N \rangle, \ 
\dR f_{\ast}S_i[1] \in \Coh_0(Y). 
\end{align} 
Here $\Coh_0(Y)$ is the abelian category of zero dimensional 
coherent sheaves on $Y$, and $N$ is the 
number of irreducible components of $\Ex(f)$, 
the exceptional locus of $f$. 
We construct $\cC_{X/Y}^{0}$ by induction on 
the number of irreducible components $N$. This 
approach is convenient to 
describe generators of $\cC_{X/Y}^{0}$, and the relationship 
under blow-downs. 

When $N=0$, then $\cC_{X/Y}=\{0\}$ and the
heart $\cC_{X/Y}^{0}$ is taken to be 
the trivial one. 
Suppose that $N>0$, and 
let us consider the finite set of points, 
\begin{align*}
f(\Ex(f))=\{p_1, \cdots, p_l \}.
\end{align*}  
Since any object $E \in \cC_{X/Y}$
is supported on $\sqcup_i f^{-1}(p_i)$, 
there is a decomposition
\begin{align*}
\cC_{X/Y}=\bigoplus_{i=1}^{l} \cC_{X/Y_i}
\end{align*}
where $X \stackrel{f_i}{\to} Y_i \to Y$
is a factorization of $f$ so that 
$\Ex(f_i)=f^{-1}(p_i)$. 
We construct t-structures on 
each $\cC_{X/Y_i}$, and take their
direct sum to construct the t-structure on $\cC_{X/Y}$. 
Since $f_i(\Ex(f_i))$ is a point, 
we may 
assume that $l=1$. 

Let 
\begin{align*}
h \colon Y' \to Y
\end{align*} be the
 blowing up at $f(\Ex(f))=\{p\}$, and $C \subset Y'$ the 
exceptional locus of $h$. 
The birational morphism 
$f \colon X \to Y$ factors through $h \colon Y' \to Y$, 
\begin{align*}
f \colon X \stackrel{g}{\to} Y' \stackrel{h}{\to}Y. 
\end{align*}
The functor $\dR f_{!}$ also factors as
\begin{align*}
\dR f_{!} \colon D^b \Coh(X) \stackrel{\dR g_{!}}{\to} D^b \Coh(Y')
\stackrel{\dR h_{!}}{\to} D^b \Coh(Y). 
\end{align*}
Therefore we have the sequence of exact functors
\begin{align}\label{C:ext}
\cC_{X/Y'} \to 
\cC_{X/Y} \stackrel{\dR g_{!}}{\to} \cC_{Y'/Y}
\end{align}
where the left functor is the natural inclusion. 
The functors $\dL g^{\ast}, g^{!}$ satisfy
\begin{align*}
\dR g_{!} \dL g^{\ast}E \cong E, \ 
\dR g_{!} g^{!} E \cong E
\end{align*}
for any $E \in D^b \Coh(Y')$. 
This implies that
$\dL g^{\ast}$ and $g^{!}$ 
induce the right and the left 
adjoint functors of 
\begin{align*}
\dR g_{!} \colon \cC_{X/Y} \to \cC_{Y'/Y}
\end{align*}
respectively. 
From this fact, it is straightforward to check that 
the sequence (\ref{C:ext})
is an exact triple as in Subsection~\ref{subsec:glue}. 

By~\cite[Lemma~3.1]{Br1},  
the standard t-structure on $D^b \Coh(Y')$
induces a bounded t-structure on $\cC_{Y'/Y}$. 
The heart is described by (cf.~\cite[Proposition~3.5.8]{MVB})
\begin{align}\label{heart:C}
\cC_{Y'/Y} \cap \Coh(Y')
 = \langle \oO_{C} \rangle. 
\end{align}
On the other hand, by the inductive assumption, 
we have the heart $\cC_{X/Y'}^{0} \subset \cC_{X/Y'}$
written as 
\begin{align}\label{heart:C2}
\cC_{X/Y'}^{0}=
\langle S'_1, \cdots, S'_{N-1} \rangle
\end{align}
for some objects $S'_j \in \cC_{X/Y'}^{0}$ with 
$1\le j \le N-1$ satisfying 
$\dR g_{\ast} S'_j[1] \in \Coh_0(Y')$.
By gluing the t-structures with hearts (\ref{heart:C}), (\ref{heart:C2})
via the exact triple (\ref{C:ext}),  
we obtain the heart
\begin{align*}
\widetilde{\cC}_{X/Y}^{0} \subset \cC_{X/Y}. 
\end{align*} 
Let us set $\widehat{C} \cneq g^{\ast} C$, where 
$g^{\ast}$ means the total pull-back. 
We naturally regard $\widehat{C}$ as a subscheme of 
$X$. 
We have the following lemma: 
\begin{lem}\label{lem:C=CS}
We have 
\begin{align}\label{C=CS}
\widetilde{\cC}_{X/Y}^{0} =
\langle \cC_{X/Y'}^{0}, \oO_{\widehat{C}}
  \rangle
\end{align}
such that $(\cC_{X/Y'}^{0}, 
\langle \oO_{\widehat{C}} \rangle)$
is a torsion pair on $\widetilde{\cC}_{X/Y}^{0}$. 
\end{lem}
\begin{proof}
We first check that the RHS is contained in the LHS. 
By the definition of gluing, it 
is obvious that $\cC_{X/Y'}^{0}$ is contained in the LHS. 
Also since $\oO_{\widehat{C}}=\dL g^{\ast} \oO_C$,
we have $\dR g_{!} \oO_{\widehat{C}}=\oO_C \in \Coh(Y')$. 
We have 
\begin{align}\notag
\Hom(\cC_{X/Y'}, \oO_{\widehat{C}})
&\cong \Hom(\dR g_{!} \cC_{X/Y'}, \oO_C) \\
\label{vanish:HdRc}
&\cong 0
\end{align}
since $\dR g_{!} \cC_{X/Y'}=0$, and 
\begin{align}\notag
\Hom(\oO_{\widehat{C}}, \cC_{X/Y'}^{> 0}) &\cong 
\Hom(\oO_{C}, \dR g_{\ast} \cC_{X/Y'}^{> 0}) \\
\notag
&\cong 0
\end{align}
since $\dR g_{\ast} \cC_{X/Y'}^{> 0} \in D^{>1}(\Coh(Y))$
by the inductive assumption. 
These imply that $\oO_{\widehat{C}}$ is contained in the LHS. 

Conversely, let us take an object $E\in \widetilde{\cC}_{X/Y}^{0}$. 
By the adjointness, we have the distinguished triangle
\begin{align*}
F \to E \to \dL g^{\ast} \dR g_{!} E. 
\end{align*}
Note that we have 
\begin{align*}
\dL g^{\ast} \dR g_{!}E \in 
\langle \oO_{\widehat{C}}
\rangle, \ F \in \cC_{X/Y'}.
\end{align*}
 Moreover, 
since $\cC_{X/Y'}^{0} \subset 
\widetilde{\cC}_{X/Y}^{0}$, 
 we have  
\begin{align*}
\hH_{\cC_{X/Y'}^{0}}^{i}(F) 
\cong \hH_{\widetilde{\cC}_{X/Y}^{0}}^{i}(F)
\end{align*}
for all $i$. 
Therefore we have the exact sequence in 
$\widetilde{\cC}_{X/Y}^{0}$
\begin{align*}
0 \to \hH_{\cC_{X/Y'}^{0}}^{0}(F) \to
E \to \dL g^{\ast} \dR g_{!} E
\to \hH_{\cC_{X/Y'}^{0}}^{1}(F) \to 0
\end{align*}
and $\hH_{\cC_{X/Y'}^{0}}^{i}(F)=0$
for $i\neq 0, 1$. 
On the other hand, 
for any $A \in \Coh(Y')$ and $A' \in \cC_{X/Y'}^{0}$, we
have 
\begin{align*}
\Hom(\dL g^{\ast}A, A') &\cong 
\Hom(A, \dR g_{\ast}A') \\
&\cong 0 
\end{align*}
since $\dR g_{\ast}A' \in \Coh_{0}(Y')[-1]$. 
Therefore we have $\hH_{\cC_{X/Y'}^{0}}^{1}(F) \cong 0$, 
$F \in \cC_{X/Y'}^{0}$ and an exact 
sequence in $\widetilde{\cC}_{X/Y}^{0}$
\begin{align}\label{FEdR}
0 \to F \to E \to \dL g^{\ast} \dR g_{!}E \to 0. 
\end{align}
This implies that $E$ is contained in the RHS of (\ref{C=CS}). 
Together with the vanishing (\ref{vanish:HdRc}), 
the exact sequence (\ref{FEdR}) implies that 
$(\cC_{X/Y'}^{0}, \langle \oO_{\widehat{C}} \rangle)$
is a torsion pair on $\widetilde{\cC}_{X/Y}^{0}$. 
\end{proof}
We also have the following lemma: 
\begin{lem}\label{lem:additional}
There is a torsion pair on $\widetilde{\cC}_{X/Y}^{0}$ of the form 
\begin{align}\label{tpair}
( \langle \oO_{\widehat{C}} \rangle, 
\oO_{\widehat{C}}^{\cC, \perp}  ),
\end{align}
where $\oO_{\widehat{C}}^{\cC, \perp}$ is the right 
orthogonal 
complement of $\oO_{\widehat{C}}$ in 
$\widetilde{\cC}_{X/Y}^{0}$. 
\footnote{We put ``$\cC$'' in the notation of 
the right orthogonal complement, in order
to distinguish it with a similar 
orthogonal complement in other abelian category 
in Subsection~\ref{subsec:relation}.}
\end{lem}
\begin{proof}
By the inductive assumption, 
the abelian category $\cC_{X/Y'}^{0}$
is the extension closure of some finite 
number of objects. 
Hence by Lemma~\ref{lem:C=CS},
it follows that 
the abelian category 
$\widetilde{\cC}_{X/Y}^{0}$ is
also 
the extension closure of some finite number of objects. 
In particular it is a noetherian abelian category. 
Hence it is enough to check that $\langle \oO_{\widehat{C}} \rangle$
is closed under quotients (cf.~\cite[Lemma~2.15 (i)]{Tcurve2}).
To prove the latter statement, note that 
$\langle \oO_{\widehat{C}} \rangle$
is closed under subobjects since it is a free part of some torsion 
pair by Lemma~\ref{lem:C=CS}.
Also since the self extension of $\oO_{\widehat{C}}$ vanishes, 
any object in $\langle \oO_{\widehat{C}} \rangle$ is a 
direct sum of $\oO_{\widehat{C}}$. 
Let us take an exact sequence in $\widetilde{\cC}_{X/Y}^{0}$,  
\begin{align*}
0 \to F \to \oO_{\widehat{C}}^{\oplus m} \to G \to 0. 
\end{align*}
By the argument above, $F$ is isomorphic to 
$\oO_{\widehat{C}}^{\oplus l}$ for some $l$. 
Then the object $G$ must be isomorphic to $\oO_{\widehat{C}}^{\oplus m-l}$, 
proving that $\langle \oO_{\widehat{C}} \rangle$ is closed under 
quotients. 
\end{proof}

By taking the tilting
with respect to the torsion pair (\ref{tpair}), 
we define the heart of a bounded t-structure 
$\cC_{X/Y}^{0}$
on $\cC_{X/Y}$ to be, 
\begin{align}\label{def:CXY}
\cC_{X/Y}^{0} \cneq \langle \oO_{\widehat{C}}^{\cC, \perp}, 
\oO_{\widehat{C}}[-1] \rangle. 
\end{align}
\begin{lem}\label{lem:CXY}
We have 
\begin{align}\label{CSN}
\cC_{X/Y}^{0}= \langle S_1, \cdots, S_{N-1}, S_{N} \rangle 
\end{align}
where $S_{N}=\oO_{\widehat{C}}[-1]$ and 
$S_i$ for $1\le i \le N-1$ is given by the universal 
extension in $\widetilde{\cC}_{X/Y}^{0}$
\begin{align}\label{univ:S}
0 \to S'_i \to S_i \to 
\oO_{\widehat{C}} \otimes \Ext^1(\oO_{\widehat{C}}, S_i') \to 0.
\end{align}
\end{lem}
\begin{proof}
We first note the vanishing 
\begin{align}\label{vanish:OOR}
\Hom(\oO_{\widehat{C}}, \cC_{X/Y'}^{0})=0
\end{align}
since $\dR g_{\ast} \cC_{X/Y'}^{0} \subset \Coh_0(Y')[-1]$. 
By the vanishing (\ref{vanish:OOR}), we have 
$\Hom(\oO_{\widehat{C}}, S_i')=0$
for $1\le i\le N-1$. 
Combined with the fact that (\ref{univ:S}) is the universal 
extension, it follows that 
$\Hom(\oO_{\widehat{C}}, S_i)=0$, 
i.e. $S_i \in \oO_{\widehat{C}}^{\cC, \perp}$
for $1\le i\le N-1$. Therefore 
the RHS of (\ref{CSN}) 
is contained in the LHS of (\ref{CSN}).

Conversely, let us take an object $E \in \oO_{\widehat{C}}^{\cC, \perp}$. 
By Lemma~\ref{lem:C=CS}, there is an exact sequence in 
$\widetilde{\cC}_{X/Y}^{0}$
\begin{align*}
0 \to F \to E \to \oO_{\widehat{C}} \otimes V \to 0
\end{align*}
for some $F\in \cC_{X/Y'}^{0}$ and some finite
dimensional $\mathbb{C}$-vector space $V$. 
Since $\Hom(\oO_{\widehat{C}}, E)=0$, we have the injection 
\begin{align*}
V \hookrightarrow \Ext^1(\oO_{\widehat{C}}, F). 
\end{align*}
Let $W$ be the cokernel of the above injection. 
There is an exact sequence in $\cC_{X/Y}^{0}$
\begin{align*}
0 \to E \to \widehat{F} \to \oO_{\widehat{C}} \otimes W \to 0
\end{align*}
where  $\widehat{F}$ is 
the universal extension in $\widetilde{\cC}_{X/Y}^{0}$ 
\begin{align}\label{univ}
0 \to F \to \widehat{F}
 \to \oO_{\widehat{C}} \otimes \Ext^1(\oO_{\widehat{C}}, F)
 \to 0.
\end{align}
It is enough to show that $\widehat{F}$
 is contained in the RHS of (\ref{CSN}). 
Since $\cC_{X/Y'}^{0}$ is the extension closure of $S_1', \cdots, 
S_{N-1}'$, 
this follows from the following claim: 
for an exact sequence in $\cC_{X/Y'}^{0}$
\begin{align}\label{M12}
0 \to F_1 \to F \to F_2 \to 0
\end{align}
suppose that 
their universal extensions
$\widehat{F}_i$ in $\widetilde{\cC}_{X/Y}^{0}$
\begin{align*}
0 \to F_i \to \widehat{F}_i \to \oO_{\widehat{C}}\otimes
\Ext^1(\oO_{\widehat{C}}, F_i) \to 0
\end{align*}
are contained in the RHS of (\ref{CSN}). 
Then $\widehat{F}$ is contained in the RHS of (\ref{CSN}). 
To prove this claim, first note that 
$\Hom(\oO_{\widehat{C}}, F_2)=0$
by the vanishing (\ref{vanish:OOR}).
Therefore applying $\Hom(\oO_{\widehat{C}}, \ast)$ to the 
sequence (\ref{M12}),
we obtain the exact sequence
\begin{align*}
0 \to \Ext^1(\oO_{\widehat{C}}, F_1) \to \Ext^1(\oO_{\widehat{C}}, F) 
\stackrel{\psi}{\to} \Ext^1(\oO_{\widehat{C}}, F_2). 
\end{align*} 
It follows  that there is an exact sequence in $\cC_{X/Y}^{0}$
\begin{align}\label{M1b}
0 \to \widehat{F}_1 \to \widehat{F} \to \overline{F}_2 \to 0
\end{align}
where $\overline{F}_2$ fits into the exact sequence in 
$\widetilde{\cC}_{X/Y}^{0}$
\begin{align*}
0 \to F_2 \to \overline{F}_2 \to
 \oO_{\widehat{C}} \otimes \Imm \psi \to 0.
\end{align*}
We have the exact sequence in $\cC_{X/Y}^{0}$
\begin{align*}
0 \to \oO_{\widehat{C}}\otimes \Cok(\psi)[-1] 
 \to \overline{F}_2 \to \widehat{F}_2 \to 0.
\end{align*}
Since $\widehat{F}_2$ is contained in the 
RHS of (\ref{CSN}), the object
$\overline{F}_2$ is also 
contained in the RHS of (\ref{CSN}). 
Combined with that $\widehat{F}_1$
is contained in the 
RHS of (\ref{CSN}), the exact sequence (\ref{M1b})
implies that the object 
$\widehat{F}$ is also contained in the
RHS of (\ref{CSN}). 
\end{proof}
Moreover we have the following lemma: 
\begin{lem}
For the objects $S_i$ in Lemma~\ref{lem:CXY}, we have
\begin{align}\label{RS0}
\dR f_{\ast} S_i[1] \in \Coh_0(Y), \ 1 \le i \le N. 
\end{align}
\end{lem}
\begin{proof}
The claim for $i=N$ is obvious. 
Suppose that $1\le i\le N-1$. 
Applying $\dR g_{\ast}$ to the sequence (\ref{univ:S}), 
we obtain the distinguished triangle
\begin{align*}
\dR g_{\ast}S'_i \to \dR g_{\ast}S_i \to \oO_{C}^{\oplus m_i}
\end{align*}
where $m_i = \dim \Ext^1(\oO_{\widehat{C}}, S_i')$. 
Since $\dR g_{\ast}S'_i \cong Q_i[-1]$ for
some zero dimensional sheaf $Q_i$ on $Y'$, 
the object $\dR g_{\ast}S_i$ is isomorphic to the two 
term complex
\begin{align*}
(\oO_C^{\oplus m_i} \stackrel{\phi}{\to} Q_i)
\end{align*}
with $\oO_C^{\oplus m_i}$ located in degree zero.
It is enough to check $h_{\ast} \Ker(\phi)=0$, 
which is equivalent to $H^0(C, \Ker(\phi))=0$. 
If $H^0(C, \Ker(\phi))$ is non-zero, then there is
a non-zero section
$s\in H^0(C, \oO_C^{\oplus m_i})$
satisfying $\phi \circ s =0$. 
By adjunction, there is non-zero
$\widehat{s} \in H^0(\widehat{C}, \oO_{\widehat{C}}^{\oplus m_i})$
such that the composition 
\begin{align*}
\oO_{\widehat{C}} \stackrel{\widehat{s}}{\to}
\oO_{\widehat{C}}^{\oplus m_i} \to S'_i[1]
\end{align*}
is zero. Here the right morphism
is induced by
the extension (\ref{univ:S}).
This contradicts the fact that (\ref{univ:S}) is the universal 
extension. Hence $h_{\ast}\Ker(\phi)=0$, 
and the condition (\ref{RS0}) holds.  
\end{proof} 
By the above lemmas, the heart $\cC_{X/Y}^{0} \subset \cC_{X/Y}$
satisfies the desired property (\ref{desire:S}). 
As a summary, we have obtained the following proposition: 
\begin{prop}\label{prop:Ctst}
Let $X$ be a smooth projective surface. 
Then for each smooth projective surface $Y$ 
and a birational morphism $f\colon X \to Y$, we 
can associated the heart of a bounded t-structure 
$\cC_{X/Y}^{0} \subset \cC_{X/Y}$ satisfying the 
following conditions: 
\begin{itemize}
\item For any $F\in \cC_{X/Y}^{0}$, the object
$\dR f_{\ast}F[1]$ is a zero dimensional sheaf on $Y$. 
\item $\cC_{X/Y}^{0}$ is the extension closure of a
finite number of objects in $\cC_{X/Y}^{0}$. 
\item 
Suppose that $f(\Ex(f))$
is a point $p\in Y$, and 
take the factorization
\begin{align*}
f \colon X\stackrel{g}{\to} Y' \stackrel{h}{\to} Y
\end{align*}
where $h$ is a 
blow-up at $p$ which contracts 
a $(-1)$-curve 
$C$ on $Y'$, and 
$\cC_{X/Y'}^{0}$ is the extension closure of objects
$S_1', \cdots, S_{N-1}'$. 
Then $\cC_{X/Y}^{0}$ is the extension closure
of objects $S_1, \cdots, S_{N-1}, S_N\cneq \oO_{\widehat{C}}[-1]$, 
where $\widehat{C}=g^{\ast}C$ and $S_i$ is the 
cone of the universal morphism 
\begin{align*}
S_i \to \oO_{\widehat{C}} \otimes \Ext^1(\oO_{\widehat{C}}, S_i') \to S_i'[1]. 
\end{align*}
\end{itemize}
\end{prop}

\subsection{Generators of the heart $\cC_{X/Y}^{0}$}\label{subsec:gen}
In this subsection, we give an
explicit description of the generator of $\cC_{X/Y}^{0}$. 
The description here is not canonical, 
since it depends on a choice of a factorization of
$f$ as in (\ref{compose}) below. 
However it will be useful in constructing stability conditions.
Before giving a general description, we 
look at our resulting generators in some examples:  
\begin{exam}\label{exam:C}
(i) Suppose that $f \colon X\to Y$
contracts disjoint $(-1)$-curves
$C_1, \cdots, C_N$ on $X$. 
Then we have 
\begin{align*}
\cC_{X/Y}^{0}= \langle \oO_{C_1}[-1], \cdots, \oO_{C_N}[-1] \rangle. 
\end{align*}

(ii) Suppose that the
$f \colon X \to Y$
factors as 
\begin{align*}
X=X_1 \stackrel{g_1}{\to}
 X_2 \stackrel{g_2}{\to} X_3=Y
\end{align*}
so that each
$g_i$ contracts a $(-1)$-curve
$C_i \subset X_i$, 
and $p_1=g_1(C_1)$
satisfies $p_1 \in C_2$. 
Then we have 
\begin{align*}
\cC_{X/Y}^{0}=
\langle \oO_{\overline{C}_2}(-1), \oO_{C_1+\overline{C}_2}[-1]\rangle. 
\end{align*}
Here $\overline{C}_i \subset X$ is the strict transform. 

(iii) 
Suppose that 
$f \colon X \to Y$
factors as 
\begin{align*}
X=X_1 \stackrel{g_1}{\to}
 X_2 \stackrel{g_2}{\to} X_3 \stackrel{g_3}{\to}X_4=Y
\end{align*}
so that each
$g_i$ contracts a $(-1)$-curve
$C_i \subset X_i$,
and $p_i=g_i(C_i)$
satisfies $p_1 \notin C_2$, 
$\{g_2(p_1), p_2 \} \subset C_3$. 
Then we have
\begin{align*}
\cC_{X/Y}^{0} =\langle \oO_{C_2 + \overline{C}_3}(-p_1), 
\oO_{C_1 + \overline{C}_3}(-p_2), \oO_{C_1 + C_2 + \overline{C}_3}[-1] 
\rangle.  
\end{align*}
\end{exam}
Our strategy is 
to factorize $f$ into 
a composition of contractions of 
$(-1)$-curves, and describe the 
generator of $\cC_{X/Y}^{0}$
by the induction on the number of contractions. 
We divide $(-1)$-curves which 
appear in the contractions into two types: 
a $(-1)$-curve is of type I if it 
is essentially obtained as an exceptional curve
of a blow-up of $Y$, and otherwise it is of type II. 
For instance in Example~\ref{exam:C} (iii), 
the curve $C_3$ is of type I, and $C_1, C_2$
are of type II. 
We describe the generator of $\cC_{X/Y}^0$
according to the above types of $(-1)$-curves. 

For a birational morphism $f\colon X \to Y$ 
as in the previous subsection, 
we factorize it into a composition of 
contractions of $(-1)$-curves
\begin{align}\label{compose}
X=X_1 \stackrel{g_1}{\to} X_2 \stackrel{g_2}{\to} \cdots
\stackrel{g_{N-1}}{\to} X_{N} \stackrel{g_{N}}{\to} X_{N+1}=Y. 
\end{align}
The birational morphism 
\begin{align*}
g_i \colon X_i \to X_{i+1}
\end{align*}
contracts a single $(-1)$-curve $C_i \subset X_i$
to a point $p_i \in X_{i+1}$. 
We also set 
\begin{align*}
&g_{i, j} \cneq g_{j-1} \circ \cdots \circ g_i \colon 
X_i \to X_j \\
& f_i \cneq g_{1, i} \colon X \to X_i 
\end{align*}
and $\widehat{C}_i \cneq f_i^{\ast}C_i$. 
For $j>i$, we also write $g_{i, j}(C_i)$
as $p_i \in X_j$ by abuse of notation. 
The curves $C_i$ are classified into two types: 
\begin{itemize}
\item Type I: for any $j>i$, we have 
$p_i \notin C_j$. 
\item Type II: 
there is $j>i$ so that 
$p_i \in C_j$. 
In this case, we define $\kappa(i)>i$
to be the smallest $j>i$ satisfying $p_i \in C_j$. 
\end{itemize}
If $C_i$ is of type I, we set 
$S_i =\oO_{\widehat{C}_i}[-1]$. 
If $C_i$ is of type II, we consider the 
exact sequence of sheaves on $X_i$
\begin{align}\label{Sbar}
0 \to \overline{S}_i \to g_{i, \kappa(i)}^{\ast}\oO_{C_{\kappa(i)}}
\to \oO_{C_i} \to 0
\end{align}
and set $S_i= \dL f_i^{\ast}\overline{S}_i (=f_i^{\ast} \overline{S}_i)$. 
Here (\ref{Sbar}) is obtained 
by restricting $g_{i, \kappa(i)}^{\ast}\oO_{C_{\kappa(i)}}$
to $C_i$, and taking its kernel. 
The sheaf $\overline{S}_i$ is written as
\begin{align*}
\overline{S}_i=
 \oO_{\overline{g_{i, \kappa(i)}^{\ast}C_{\kappa(i)} -C_i}}(-p_i)
\end{align*}
for $p_i \in C_{\kappa(i)}$. 
\begin{prop}\label{prop:gen}
In the above notation, we have
\begin{align*}
\cC_{X/Y}^{0}= \langle S_1, \cdots, S_N \rangle. 
\end{align*}
\end{prop}
\begin{proof}
We show the proposition by the induction on $N$. 
Suppose that the claim 
holds for $f_N \colon X \to X_N$. 
Then we have 
\begin{align*}
\cC_{X/X_N}^{0}=\langle S'_1, \cdots, S'_{N-1} \rangle
\end{align*}
where $S'_i$ are the objects defined similarly to 
$S_i$, applied for the composition
\begin{align*}
X=X_1 \stackrel{g_1}{\to} X_2 \stackrel{g_2}{\to} \cdots
\stackrel{g_{N-1}}{\to} X_{N}.
\end{align*} 
Noting
Proposition~\ref{prop:Ctst} and 
$S_N=\oO_{\widehat{C}_{N}}[-1]$, 
it is enough to show that there 
is a distinguished triangle 
for each $1\le i\le N-1$,
\begin{align}\label{dist:S}
S_i \to \oO_{\widehat{C}_N} \otimes 
\Ext^1(\oO_{\widehat{C}_N}, S'_i) \to S'_i[1]. 
\end{align}
For $1\le i\le N-1$, we have 
the following three cases: 
\begin{case}
$C_i$ is of type I for both of $X\to X_N$
and $X\to Y$. 
\end{case}
In this case, we have 
$S'_i=S_i=\oO_{\widehat{C}_i}[-1]$.
Also we have 
\begin{align*}
\Ext^1(\oO_{\widehat{C}_N}, S'_i)
&= \Hom(\oO_{\widehat{C}_N}, \oO_{\widehat{C}_i}) \\
&= \Hom_{X_N}(\oO_{C_N}, \oO_{p_i}) \\
&\cong 0
\end{align*}
since $p_i \notin C_N$. 
Therefore we have the distinguished triangle (\ref{dist:S}). 
\begin{case}
$C_i$ is of type I for $X\to X_N$ and type II
for $X\to Y$. 
\end{case}
In this case, we have 
$S'_i=\oO_{\widehat{C}_i}[-1]$, 
$\kappa(i)=N$ and $S_i=\dL f_{i}^{\ast} \overline{S}_i$.  
We have 
\begin{align*}
\Ext^1(\oO_{\widehat{C}_N}, S_i') &\cong
\Hom_{X_N}(\oO_{C_N}, \oO_{p_i}) \\
&\cong \mathbb{C}
\end{align*}
since $p_i \in C_N$. 
By pulling back the exact sequence (\ref{Sbar}) 
to $X$ via $f_i$, we have the distinguished triangle (\ref{dist:S}). 
\begin{case}
 $C_i$ is of type II for both of $X\to X_N$
and $X\to Y$. 
\end{case}
In this case, $1\le \kappa(i) \le N-1$
and $S'_i=S_i =\dL f_i^{\ast} \overline{S}_i$. 
We have
\begin{align}\notag
\Ext^1(\oO_{\widehat{C}_N}, S_i') &\cong
\Ext^1_{X_i}(\dL g_{i, N}^{\ast}\oO_{C_N}, \overline{S}_i) \\
\label{vanish}
&\cong \Ext^1_{X_N}(\oO_{C_N}, \dR g_{i, N \ast} \overline{S}_i). 
\end{align}
Applying $\dR g_{i, \kappa(i) \ast}$ to the sequence (\ref{Sbar}), 
we obtain the distinguished triangle
\begin{align*}
\dR g_{i, \kappa(i) \ast}\overline{S}_i \to 
\oO_{C_{\kappa(i)}} \to \oO_{p_i}
\end{align*}
such that the right morphism is non-trivial 
since $p_i \in C_{\kappa(i)}$. Therefore 
we have $\dR g_{i, \kappa(i) \ast} 
\overline{S}_i \cong \oO_{C_{\kappa(i)}}(-1)$
and 
\begin{align*}
\dR g_{i, N \ast}\overline{S}_i &\cong 
\dR g_{ \kappa(i), N \ast}\oO_{C_{\kappa(i)}}(-1) \\
&\cong 0. 
\end{align*}
Therefore (\ref{vanish}) vanishes and 
we have the distinguished triangle (\ref{dist:S}). 
\end{proof}

\begin{rmk}
By the construction of $S_i$, we 
obviously obtain the generators of $\cC_{X/Y}^0$
 in Example~\ref{exam:C}. 
\end{rmk}

\subsection{t-structures on $D^b \Coh(X)$}
Let $f\colon X \to Y$ be a birational morphism 
as in the previous subsections. 
Let
\begin{align*}
\aA_Y \subset D^b \Coh(Y)
\end{align*}
be the heart of a bounded t-structure such that 
$\oO_y \in \aA_Y$ for any $y\in Y$. 
We construct the heart of a t-structure on $D^b \Coh(X)$
by gluing $\aA_Y$ and $\cC_{X/Y}^{0}$ constructed in the 
previous subsections. 

Let us consider the following sequence of 
exact functors
\begin{align*}
\cC_{X/Y} \to D^b \Coh(X) \stackrel{\dR f_{!}}{\to}
D^b \Coh(Y)
\end{align*}
where the left functor is the natural inclusion. 
It is straightforward to check 
that the above sequence is
an exact triple as in Subsection~\ref{subsec:glue}. 
By gluing $\aA_Y$ and $\cC_{X/Y}^{0}$, we obtain
the heart
\begin{align*}
\aA_X \subset D^b \Coh(X). 
\end{align*}
The heart $\aA_X$ is described as follows:
\begin{lem}\label{lem:ACA}
We have
\begin{align}\label{ACA}
\aA_X =\langle \cC_{X/Y}^{0}, \dL f^{\ast} \aA_{Y}\rangle
\end{align}
and $(\cC_{X/Y}^{0}, \dL f^{\ast} \aA_Y)$ is 
a torsion pair on $\aA_X$. 
\end{lem}
\begin{proof}
The proof is very similar to Lemma~\ref{lem:C=CS}. 
First we show that the RHS of (\ref{ACA}) is contained 
in the LHS of (\ref{ACA}). 
It is obvious that $\cC_{X/Y}^{0}$ is contained in the 
LHS, so we show that $\dL f^{\ast} \aA_Y$ is contained in
the LHS. 
For $M \in \aA_Y$, we have 
$\dR f_{!} \dL f^{\ast}M \cong M \in \aA_Y$
and 
\begin{align}\label{check:ad}
\Hom(\cC_{X/Y}, \dL f^{\ast}M) \cong 0
\end{align}
 by the adjunction. Also we have 
\begin{align}\notag
\Hom(\dL f^{\ast}M, \cC_{X/Y}^{\ge 0}) &\cong 
\Hom(M, \dR f_{\ast} \cC_{X/Y}^{\ge 0}) \\
\label{vanish2}
&\cong 0
\end{align}
since $\dR f_{\ast} \cC_{X/Y}^{\ge 0} \subset D^{>0} \Coh(Y)$
with cohomology sheaves zero dimensional,
and $\Coh_0(Y) \subset \aA_Y$. 
Therefore $\dL f^{\ast} M$ is an object in 
$\aA_X$ by the definition of the gluing. 

Conversely, we show that $\aA_X$ is contained in the 
RHS of (\ref{ACA}). 
For an object $E\in \aA_X$, there is a 
distinguished triangle
\begin{align*}
F \to E \to \dL f^{\ast} \dR f_{!}E
\end{align*}
with $F \in \cC_{X/Y}$. 
Similarly to the proof of Lemma~\ref{lem:C=CS}, we have the 
exact sequence in $\aA_X$
\begin{align*}
0 \to \hH_{\cC_{X/Y}^{0}}^{0}(F) \to E \to \dL f^{\ast} \dR f_{!}E
\to \hH_{\cC_{X/Y}^{0}}^{1}(F) \to 0
\end{align*}
and $\hH_{\cC_{X/Y}^{0}}^{i}(F)=0$ for $i\neq 0, 1$. 
By the vanishing (\ref{vanish2}), 
we also have $\hH_{\cC_{X/Y}^{0}}^1(F)=0$ and 
$F \in \cC_{X/Y}^{0}$. Consequently we have the 
exact sequence in $\aA_{X}$
\begin{align}\label{FEdLf}
0 \to F \to E \to \dL f^{\ast} \dR f_{!}E \to 0
\end{align}
with $F \in \cC_{X/Y}^{0}$. 
Therefore $E$
is contained in the LHS of (\ref{ACA}).  
By (\ref{check:ad}) and (\ref{FEdLf}), 
the pair $(\cC_{X/Y}^{0}, \dL f^{\ast} \aA_Y)$
is a torsion pair on $\aA_X$. 
\end{proof}
\section{Proof of Theorem~\ref{thm:intro}}\label{sec:proof}
In this section, we construct 
a connected open subset 
\begin{align*}
U(Y) \subset \Stab(X)_{\mathbb{R}}
\end{align*}
for each birational morphism $f\colon X \to Y$, 
and prove Theorem~\ref{thm:intro}. 
In what follows, we always assume that 
$f \colon X \to Y$ is a birational
morphism between smooth projective surfaces. 

\subsection{Central charges corresponding to $U(Y)$}
Let
\begin{align*}
\mathrm{NS}_f(X)_{\mathbb{R}} \subset \mathrm{NS}(X)_{\mathbb{R}}
\end{align*}
be the orthogonal complement of 
$f^{\ast} \mathrm{NS}(Y)$ with respect to the 
intersection pairing. 
Note that $\mathrm{NS}_f(X)_{\mathbb{R}}$ is a linear subspace
of $\mathrm{NS}(X)_{\mathbb{R}}$
spanned by the irreducible components of the exceptional 
locus of $f$. 
For fixed $k>0$, we set 
\begin{align*}
C_{f, k}(X) \cneq \left\{ D \in \mathrm{NS}_f(X)_{\mathbb{R}} :
\begin{array}{l}
D \cdot c_1(F)>0 \mbox{ for all } \\
F \in \cC_{X/Y}^{0}, \ 
D^2 + k >0.
\end{array} \right\}. 
\end{align*}
We have the following lemma:
\begin{lem}\label{lem:Cf}
$C_{f, k}(X)$ is a non-empty connected
open subset of $\mathrm{NS}_f(X)_{\mathbb{R}}$. 
\end{lem}
\begin{proof}
We factorize $f\colon X \to Y$ into the
composition of blow-downs as in (\ref{compose}). 
In the notation of Subsection~\ref{subsec:gen}, 
we have 
\begin{align*}
\mathrm{NS}_f(X)_{\mathbb{R}}= \bigoplus_{i=1}^{N} \mathbb{R}
[\widehat{C}_i]
\end{align*}
for $\widehat{C}_i=f_i^{\ast}C_i$. 
For $D \in \mathrm{NS}_f(X)_{\mathbb{R}}$, 
it is contained in $C_{f, k}(X)$ if and only if 
$D \cdot c_1(S_i)>0$ for all $1\le i\le N$, 
where $S_i$ is given in Subsection~\ref{subsec:gen}, 
and $D^2 +k>0$. 
If we write $D\in \mathrm{NS}_f(X)_{\mathbb{R}}$ as 
\begin{align*}
D=\sum_{i=1}^{N} t_i [\widehat{C}_i]
\end{align*}
for $t_i \in \mathbb{R}$, then 
$D \cdot c_1(S_i)$ is calculated as  
\begin{align*}
D \cdot c_1(S_i)= \left\{ \begin{array}{cc}
t_i, & i \mbox{ is of type I} \\
t_i - t_{\kappa(i)}, & i \mbox{ is of type II}. 
\end{array} \right. 
\end{align*}
Therefore $C_{f, k}(X)$ is identified with  
\begin{align*}
C_{f, k}(X)= \left\{ (t_1, \cdots, t_N) \in \mathbb{R}^N : 
\begin{array}{l}
t_i >0, \  i \mbox{ is of type I} \\
t_i >t_{\kappa(i)}, \ i \mbox{ is of type II} \\
t_1^2 + \cdots + t_N^2 <k.  
\end{array}
\right\}. 
\end{align*}
Hence $C_{f, k}(X)$ is a non-empty connected
open subset of $\mathrm{NS}_f(X)_{\mathbb{R}}$. 
\end{proof}
We consider the following sets
\begin{align*}
A^{\dag}(Y) &\cneq 
\{f^{\ast}\omega  +D: \omega \in A(Y), \ 
D \in C_{f, \omega^2}(X) \} \\
\overline{A}^{\dag}(Y) &\cneq 
\{f^{\ast}\omega  +D: \omega \in \overline{A}(Y), \ 
D \in C_{f, \omega^2}(X) \}. 
\end{align*} 
The set $A^{\dag}(Y)$
is a topological fiber bundle 
\begin{align*}
f_{\ast} \colon 
A^{\dag}(Y) \to A(Y)
\end{align*}
whose fiber at $\omega$
is $C_{f, \omega^2}(X)$. 
By Lemma~\ref{lem:Cf}, 
$A^{\dag}(Y)$ 
is an open 
connected subset of $\mathrm{NS}(X)_{\mathbb{R}}$, 
and $\overline{A}^{\dag}(Y)$ is its partial 
compactification. 
We will 
consider the central charges
of the form  
\begin{align*}
Z_{f^{\ast}\omega+D} \in N(X)_{\mathbb{C}}^{\vee}, 
\quad  f^{\ast}\omega + D \in \overline{A}^{\dag}(Y).
\end{align*}
The compatible t-structure will be given in the 
next subsection.

\subsection{t-structures corresponding to $U(Y)$}
For a rational point
 $\omega \in \overline{A}(Y)$, 
we have the heart of a bounded t-structure 
\begin{align*}
\aA_{\omega} \subset D^b \Coh(Y)
\end{align*}
constructed in Subsection~\ref{tilt:Y}. 
By the construction, 
 all the objects $\oO_y$ 
for $y\in Y$ are
contained in $\aA_{\omega}$. 
Therefore Lemma~\ref{lem:ACA}
implies the existence of a bounded t-structure
on $D^b \Coh(X)$ with heart given by 
\begin{align}\label{tAXY}
\aA_{\omega}(X/Y) \cneq 
\langle \cC_{X/Y}^{0}, \dL f^{\ast} \aA_{\omega} \rangle
\end{align}
such that $(\cC_{X/Y}^{0}, \dL f^{\ast} \aA_{\omega})$
is a torsion pair on $\aA_{\omega}(X/Y)$. 
Later we will need the following lemma: 
\begin{lem}\label{closed:under}
The subcategories 
\begin{align*}
\cC_{X/Y}^{0}, \dL f^{\ast} \aA_{\omega}
 \subset \aA_{\omega}(X/Y)
\end{align*}
are closed under subobjects and quotients. 
\end{lem}
\begin{proof}
Since $(\cC_{X/Y}^{0}, \dL f^{\ast} \aA_{\omega})$
is a torsion pair on $\aA_{\omega}(X/Y)$, 
the subcategory $\cC_{X/Y}^{0}$ is closed under quotients, 
and the subcategory $\dL f^{\ast} \aA_{\omega}$ is closed 
under subobjects. For $F \in \cC_{X/Y}^{0}$, 
suppose that $A \hookrightarrow F$ is an injection
in $\aA_{\omega}(X/Y)$. Then it induces an injection
$\dR f_{!}A \hookrightarrow \dR f_{!}F$ in
$\aA_{\omega}$. Since $\dR f_{!}F=0$, we have 
$\dR f_{!}A=0$, hence $A \in \cC_{X/Y}^{0}$. 
This implies that $\cC_{X/Y}^{0}$ is also 
closed under subobjects. 

For $M \in \aA_{\omega}$, let us take an exact sequence in 
$\aA_{\omega}(X/Y)$
\begin{align*}
0 \to E_1 \to \dL f^{\ast}M \to E_2 \to 0. 
\end{align*}
As we observed before, we have $E_1 \in \dL f^{\ast} \aA_{\omega}$. 
For any $F \in \cC_{X/Y}^{0}$, we have  
\begin{align*}
\dR \Hom(F, \dL f^{\ast} \aA_{\omega})=0
\end{align*}
since $\dR f_{!}F=0$. Therefore we have 
$\Hom(F, E_2)=0$, 
hence $E_2 \in \dL f^{\ast}\aA_{\omega}$
since $(\cC_{X/Y}^{0}, \dL f^{\ast} \aA_{\omega})$
is a torsion pair on $\aA_{\omega}(X/Y)$. 
This implies that $\dL f^{\ast}\aA_{\omega}$ is 
also closed under quotients. 
\end{proof}

We will also need the following lemma:
\begin{lem}\label{lem:noether}
The abelian category $\aA_{\omega}(X/Y)$ is noetherian.
\end{lem}
\begin{proof}
Suppose that there is an infinite sequence of 
surjections in $\aA_{\omega}(X/Y)$
\begin{align}\label{inf:noe}
E=E_1 \twoheadrightarrow E_2 \twoheadrightarrow \cdots \twoheadrightarrow
E_i 
\twoheadrightarrow E_{i+i} \twoheadrightarrow \cdots. 
\end{align}
Applying $\dR f_{!}$ to the sequence (\ref{inf:noe}), 
we obtain surjections 
\begin{align}\label{surj:A}
\dR f_{!} E_{i} \twoheadrightarrow 
\dR f_{!} E_{i+1}
\end{align} in $\aA_{\omega} \subset D^b \Coh(Y)$. 
Since $\aA_{\omega}$ is noetherian by the proof 
of~\cite[Lemma~5.2]{Todext}, 
we may assume that (\ref{surj:A}) 
are isomorphisms for all $i$. 
Hence if we take the exact sequences in $\aA_{\omega}(X/Y)$
\begin{align*}
0 \to F_i \to E \to E_{i} \to 0 
\end{align*}
then $F_i \in \cC_{X/Y}^{0}$. On the other hand, 
we have the exact sequence in $\aA_{\omega}(X/Y)$
\begin{align*}
0 \to F \to E \to \dL f^{\ast} M \to 0
\end{align*}
for $F \in \cC_{X/Y}^{0}$ and $M\in \aA_{\omega}$. 
Since $\Hom(F_i, \dL f^{\ast}M)=0$, we have the 
sequence of injections in $\aA_{\omega}(X/Y)$
\begin{align*}
F_1 \hookrightarrow F_2 \hookrightarrow \cdots \hookrightarrow F. 
\end{align*}
By Lemma~\ref{closed:under}, the above sequence is a sequence of 
injections in $\cC_{X/Y}^{0}$. Since $\cC_{X/Y}^{0}$ is 
the extension closure of a finite number of objects, 
it is noetherian, hence the above sequence terminates. 
Therefore the sequence (\ref{inf:noe}) also terminates.  
\end{proof}

\subsection{Construction of $U(Y)$}
For $f^{\ast}\omega +D \in \overline{A}^{\dag}(Y)$
with $\omega$, $D$ rational, we 
consider the pair 
\begin{align}\label{pair:omegaD}
\sigma_{f^{\ast}\omega +D} \cneq 
(Z_{f^{\ast}\omega+D}, \aA_{\omega}(X/Y)). 
\end{align}
The purpose here is to show 
that $\sigma_{f^{\ast} \omega +D}$ gives
a point in $\Stab(X)_{\mathbb{R}}$. 
We first prepare a lemma:
let us consider the central charge 
$Z_{\omega, D} \in N(Y)^{\vee}_{\mathbb{C}}$
defined by
\begin{align*}
Z_{\omega, D}(M) \cneq
Z_{\omega}(M) + \frac{D^2}{2} \ch_0(M). 
\end{align*}
Note that, for any $M \in D^b \Coh(Y)$, it 
satisfies the following equality: 
\begin{align}\label{ZodM}
Z_{f^{\ast}\omega +D}(\dL f^{\ast} M)
= Z_{\omega, D}(M). 
\end{align}
\begin{lem}\label{lem:prepare}
We have 
\begin{align*}
\sigma_{\omega, D} \cneq  
(Z_{\omega, D}, \aA_{\omega}) \in \Stab(Y). 
\end{align*}
\end{lem}
\begin{proof}
Note that the pair $(Z_{\omega}, \aA_{\omega})$
is shown to be an element of $\Stab(Y)$
in~\cite[Lemma~3.12]{Todext}, and almost the same 
argument is applied. Indeed
for a non-zero $M\in \aA_{\omega}$, 
$Z_{\omega, D}(M)$ is written as 
\begin{align}\notag
-\ch_2(M) + \frac{\ch_0(M)}{2}(D^2 +\omega^2)
+i\ch_1(M) \cdot \omega. 
\end{align}
By the construction of $\aA_{\omega}$, 
we have $\ch_1(M) \cdot \omega \ge 0$. 
Moreover, the proof of~\cite[Lemma~3.12]{Todext}
shows that if $\ch_1(M) \cdot \omega =0$, 
then $M$ satisfies either 
$\ch_0(M)<0, \ch_2(M) \ge 0$ or $\ch_0(M)=0, \ch_2(M)>0$. 
By the definition of $C_{f, \omega^2}(X)$, we have 
$D^2 + \omega^2>0$, hence 
$Z_{\omega, D}$ satisfies the property (\ref{pro:1}). 

The proofs for other properties are also 
the same as in~\cite[Lemma~3.12]{Todext}. 
Indeed the abelian category $\aA_{\omega}$ is noetherian, 
so there is no need to modify the proof for the Harder-Narasimhan 
property. As for the support property, since 
$D \in C_{f, s^2\omega^2}(X)$ for any $s\ge 1$, 
the wall-crossing method in~\cite[Theorem~3.23]{Todext} 
for the family $\{ \sigma_{s\omega, D}\}_{s\ge 1}$
works as well. 
This implies that the Chern 
characters of $Z_{\omega, D}$-semistable objects in $\aA_{\omega}$
satisfies the same 
Bogomolov-Gieseker type inequality 
as in~\cite[Theorem~3.23]{Todext},
and the same computation in the proof of~\cite[Lemma~3.12]{Todext}
shows the support property for $\sigma_{\omega, D}$. 
Since there is no need to modify the proof, we omit the detail. 
\end{proof}

Using the above lemma, we show the following: 
\begin{lem}\label{lem:Adag}
In the above situation, 
the pair (\ref{pair:omegaD}) is a stability condition on $D^b \Coh(X)$. 
\end{lem}
\begin{proof}
We first check that $\sigma_{f^{\ast}\omega +D}$
satisfies the property (\ref{pro:1}). 
For non-zero $F \in \cC_{X/Y}^{0} $ and $M \in \aA_{\omega}$, we have 
the equality (\ref{ZodM}) and 
\begin{align}\label{im:pos}
\Imm Z_{f^{\ast}\omega +D}(F) &=c_1(F) \cdot D >0.
\end{align}
Combined with Lemma~\ref{lem:prepare} and the fact 
that $\aA_{\omega}(X/Y)$ is the extension closure of 
$\cC_{X/Y}^{0}$ and $\dL f^{\ast}\aA_{\omega}$, 
it follows that $\sigma_{f^{\ast}\omega +D}$ satisfies 
the property (\ref{pro:1}). 

In order to show the Harder-Narasimhan property, 
since $\aA_{\omega}(X/Y)$ is noetherian by Lemma~\ref{lem:noether}, 
it is enough to show that
there is no
infinite sequence
\begin{align}\label{inf:seq}
E=E_1 \supset E_2 \supset 
\cdots \supset E_i \supset E_{i+1} \supset \cdots 
\end{align}
in $\aA_{\omega}(X/Y)$
such that 
\begin{align}\label{z:ineq}
\arg Z_{f^{\ast}\omega +D}(E_{i+1}) >\arg Z_{f^{\ast}\omega +D}(E_i/E_{i+1})
\end{align}
for all $i$ (cf.~\cite[Proposition~2.4]{Brs1}).
Suppose that a sequence (\ref{inf:seq})
satisfying (\ref{z:ineq}) exists. 
Since $\Imm Z_{f^{\ast}\omega +D}(\ast)$ is discrete
by the rationality of $\omega$ and $D$, 
we may assume that $\Imm Z_{f^{\ast}\omega +D}(E_i)$
is constant, hence $\Imm Z_{f^{\ast}\omega +D}(E_i/E_{i+1})=0$. 
This implies that $\arg Z_{f^{\ast}\omega +D}(E_i/E_{i+1})=\pi$, 
which contradicts to (\ref{z:ineq}). 
\end{proof}

We also have the following lemma: 
\begin{lem}\label{rel:pull}
An object $M\in \aA_{\omega}$ is 
$Z_{\omega, D}$-(semi)stable if and only if 
$\dL f^{\ast}M \in \aA_{\omega}(X/Y)$ is $Z_{f^{\ast}\omega +D}$-(semi)stable. 
\end{lem}
\begin{proof}
Since the equality (\ref{ZodM}) holds, 
the lemma obviously follows from
 Lemma~\ref{closed:under}. 
\end{proof}

The final step is to show the support property
for the pair (\ref{pair:omegaD}). We have the 
following proposition: 
\begin{prop}\label{prop:Adag}
In the above situation, we have 
\begin{align*}
\sigma_{f^{\ast}\omega +D} \in \Stab(X)_{\mathbb{R}}
\end{align*}
i.e. $\sigma_{f^{\ast}\omega +D}$ satisfies the support property. 
\end{prop}
\begin{proof}
We first note that, for any $F\in \cC_{X/Y}^{0}$, we have
\begin{align*}
Z_{f^{\ast}\omega +D}(F)= Z_{D}(F) 
\end{align*}
and its imaginary part is positive by (\ref{im:pos}).
Since $\cC_{X/Y}^{0}$ is 
the extension closure of a finite number of objects, 
there is $0<\theta \le 1$ so that 
\begin{align}\label{C:theta}
Z_{D}(\cC_{X/Y}^{0} \setminus \{0\})
\subset 
\mathbb{H}_{\theta}
\end{align}
where $\mathbb{H}_{\theta}$ is defined by 
\begin{align*}
\mathbb{H}_{\theta} 
\cneq 
\{ r \exp(i\pi \phi) : r>0, \phi \in [\theta, 1] \}. 
\end{align*}
We can find a constant $K(\theta)>0$, which 
only depends on $\theta$, satisfying 
the following: 
for any $k\ge 1$ and $z_1, \cdots, z_k \in \mathbb{H}_{\theta}$, we have
\begin{align}\label{lem:sin}
\frac{\lvert z_1 + \cdots + z_k \rvert}{\lvert z_1 \rvert + \cdots + 
\lvert z_k \rvert} \ge K(\theta). 
\end{align}
For instance, one can take $K(\theta)=\sin^2 \pi \theta/2$. 
The proof of this fact is an easy exercise, and we 
omit the proof. 

Let us take a $Z_{f^{\ast}\omega +D}$-semistable
 object $E\in \aA_{\omega}(X/Y)$. 
We 
have the exact sequence in $\aA_{\omega}(X/Y)$
\begin{align}\label{supp:FEM}
0 \to F \to E \to \dL f^{\ast} M \to 0
\end{align}
for $F \in \cC_{X/Y}^{0}$ and $M \in \aA_{\omega}$.
We find a constant $K$ as in Definition~\ref{def:sprop}
by dividing into the three cases: 
\begin{ccase}
$M=0$, i.e. $E \in \cC_{X/Y}^{0}$.  
\end{ccase}
In this case, 
let us take $K'>0$ so that the following holds: 
\begin{align*}
\frac{\lVert S_i \rVert}{\lvert Z_{D}(S_i) \rvert} <K'
\end{align*}
for all $1\le i \le N$. Here $S_1, \cdots, S_N$
are the objects in $\cC_{X/Y}^{0}$ as in Proposition~\ref{prop:gen}. 
Then by (\ref{lem:sin}), it follows that  
\begin{align*}
\frac{\lVert E \rVert}{\lvert Z_{f^{\ast}\omega +D}(E) \rvert}
< \frac{K'}{K(\theta)} (=K). 
\end{align*}
Note that the $Z_{f^{\ast}\omega +D}$-stability of $E$
is not needed in the above argument. 
\begin{ccase}
$F=0$, i.e. $E \cong \dL f^{\ast} M$. 
\end{ccase}
In this case, 
the object $M$ is $Z_{\omega, D}$-semistable 
by Lemma~\ref{rel:pull}. 
By Lemma~\ref{lem:prepare}, 
the pair $(Z_{\omega, D}, \aA_{\omega})$ satisfies the 
support property. Therefore we can find $K>0$, 
which is independent of $M$, so that 
\begin{align*}
\frac{\lVert E \rVert}{\rvert Z_{f^{\ast}\omega +D}(E) \rvert}
=\frac{\lVert M \rVert}{\lvert Z_{\omega, D} (M) \rvert} < K.
\end{align*}
\begin{ccase}
$F\neq 0$ and $M\neq 0$. 
\end{ccase}
In this case, 
note that the object $M$ may not be
$Z_{\omega, D}$-semistable. 
So there may be an exact sequence in $\aA_{\omega}$
\begin{align*}
0\to M'' \to M \to M' \to 0
\end{align*}
satisfying
\begin{align}\label{theta2}
\arg Z_{\omega, D}(M'') > \arg Z_{\omega, D}(M) > \arg Z_{\omega, D}(M'). 
\end{align}
We have the surjections in $\aA_{\omega}(X/Y)$
\begin{align*}
E \twoheadrightarrow \dL f^{\ast} M \twoheadrightarrow 
\dL f^{\ast} M'
\end{align*}
such that the kernel of their composition 
has the numerical class $[F] + [\dL f^{\ast}M'']$. 
By the $Z_{f^{\ast}\omega +D}$-stability of $E$, we have 
\begin{align}\notag
\arg (Z_{D}(F) +  Z_{\omega, D}(M'') )&= 
\label{theta3}
\arg Z_{f_{\ast}\omega +D}(F \oplus \dL f^{\ast}M'') \\
&\le \arg Z_{\omega, D}(M'). 
\end{align}
On the other hand, by 
(\ref{C:theta}), 
the exact sequence (\ref{supp:FEM}) and the 
$Z_{f^{\ast}\omega +D}$-stability of $E$, 
we have the inequalities
\begin{align}\label{theta1}
\pi \theta \le \arg Z_{D}(F)
\le \arg Z_{\omega, D}(M). 
\end{align}
The inequalities (\ref{theta2}), (\ref{theta3}) and (\ref{theta1})
imply that $\arg Z_{\omega, D}(M') \ge \pi \theta$. 
Let us take the $Z_{\omega, D}$-semistable factors of $M$, 
\begin{align*}
M_1, \cdots, M_k \in \aA_{\omega}. 
\end{align*}
Then the above argument implies that
$Z_{\omega, D}(M_i) \in \mathbb{H}_{\theta}$
for all $1\le i\le k$. 
Let $K>0$ be a constant which we took 
 in the previous cases. Then we have
\begin{align*}
\frac{\lVert E \rVert}{\lvert Z_{f^{\ast}\omega +D}(E) \rvert}
&\le \frac{1}{K(\theta)} \cdot
\frac{\lVert F \rVert + \sum_{i=1}^{k}
\lVert M_i \rVert}{\lvert Z_{D}(F) \rvert + \sum_{i=1}^{k}
\lvert Z_{D, \omega}(M_i)\rvert} \\
&\le \frac{K}{K(\theta)}
\end{align*}
by
(\ref{lem:sin}) and the results in
the previous steps. Therefore 
$\sigma_{f^{\ast}\omega +D}$ satisfies the 
support property. 
\end{proof}

For an irrational $f^{\ast}\omega +D$, we have the following 
analogue of Proposition~\ref{prop:lift}:
\begin{prop}\label{prop:liftY}
The embedding $\overline{A}^{\dag}(Y) \subset \mathrm{NS}(X)_{\mathbb{R}}$
lifts to a 
continuous map 
\begin{align}\label{cont:liftY}
\sigma_Y \colon \overline{A}^{\dag}(Y) \to \Stab(X)_{\mathbb{R}}
\end{align}
which takes any rational point 
$f^{\ast}\omega + D$ in $\overline{A}^{\dag}(Y)$
to the stability condition $\sigma_{f^{\ast}\omega +D}$
in Proposition~\ref{prop:Adag}. 
\end{prop}
\begin{proof}
The proof will be given in Subsection~\ref{subsec:techY}. 
\end{proof}
We define $U(Y)$ to 
be
\begin{align*}
U(Y) \cneq \sigma_Y(A^{\dag}(Y)) \subset \Stab(X)_{\mathbb{R}}. 
\end{align*}
Note that $U(Y)$ is a connected open subset of 
$\Stab(X)_{\mathbb{R}}$, which 
is homeomorphic to $A^{\dag}(Y)$ under the 
forgetting map 
$\Stab(X)_{\mathbb{R}} \to \mathrm{NS}(X)_{\mathbb{R}}$.

\begin{rmk}
In the situation of Example~\ref{exam:P2}, it is easy to see that
\begin{align*}
&A^{\dag}(X)=\{ x[H]+y[C] : x>0, -x<y<0\} \\
&A^{\dag}(\mathbb{P}^2)
=\{x[H] +y[C] ; x>0, 0<y<x \}. 
\end{align*}
Therefore we obtain the description in 
Example~\ref{exam:P2}. 
\end{rmk}

\subsection{Relations of $U(Y)$ under blow-downs}\label{subsec:relation}
In the situation of the previous subsections, 
suppose that $f$ factors as 
\begin{align*}
f \colon X \stackrel{g}{\to} Y' \stackrel{h}{\to} Y
\end{align*}
where $h$ 
contracts a single $(-1)$-curve 
$C$ on $Y'$ to a point in $Y$. 
The purpose of this subsection is 
to prove that $\overline{U}(Y) \cap \overline{U}(Y')$
is non-empty of real codimension one.
 
We first see the relationship between 
the hearts of bounded t-structures (\ref{tAXY})
under blow-downs. 
Let us take a rational point
$\omega \in A(Y)$, 
and consider $h^{\ast}\omega \in \overline{A}(Y')$. 
\begin{lem}
There is a torsion pair on $\aA_{h^{\ast}\omega}(X/Y')$
of the form
\begin{align}\label{t:pair}
(\langle \oO_{\widehat{C}} \rangle, \oO_{\widehat{C}}^{\aA, \perp} )
\end{align}
where 
$\widehat{C}=g^{\ast}C$ and 
$\oO_{\widehat{C}}^{\aA, \perp}$ is 
the right orthogonal complement of $\oO_{\widehat{C}}$
in $\aA_{h^{\ast}\omega}(X/Y')$.
\end{lem}
\begin{proof}
Since $\oO_C \in \PPer(Y'/Y)$, 
we have $\oO_{\widehat{C}} \in \aA_{h^{\ast}\omega}(X/Y')$. 
Also the 
abelian category $\aA_{h^{\ast}\omega}(X/Y')$
is noetherian by Lemma~\ref{lem:noether}, so 
it is enough to check that $\langle \oO_{\widehat{C}} \rangle$
is closed under quotients in $\aA_{h^{\ast}\omega}(X/Y')$
 (cf.~\cite[Lemma~2.15 (i)]{Tcurve2}).
Let us take an exact sequence in $\aA_{h^{\ast}\omega}(X/Y')$
\begin{align*}
0 \to E_1 \to \oO_{\widehat{C}}^{\oplus m} \to E_2 \to 0
\end{align*}
for $m\in \mathbb{Z}_{\ge 1}$. 
By Lemma~\ref{closed:under}, $E_i$ is of the form $\dL g^{\ast}M_i$
for some $M_i \in \aA_{h^{\ast}\omega}$, and 
we have the exact sequence in $\aA_{h^{\ast}\omega}$
\begin{align*}
0 \to M_1 \to \oO_{C}^{\oplus m} \to M_2 \to 0. 
\end{align*}
Since $\oO_C$ is a simple object in $\PPer(Y'/Y)$
(cf.~\cite[Proposition~3.5.8]{MVB}),
it 
easily follows that $\oO_C$ is also a simple object in 
$\aA_{h^{\ast}\omega}$. 
Hence $M_i \in \langle \oO_C \rangle$ and
$E_i \in \langle \oO_{\widehat{C}} \rangle$
follows. This implies that $\langle \oO_{\widehat{C}} \rangle$
is closed under quotients. 
\end{proof}

The abelian categories 
$\aA_{\omega}(X/Y)$ and $\aA_{h^{\ast}\omega}(X/Y')$
are related as follows:
\begin{lem}\label{final:tilting}
In the above situation, we have 
\begin{align*}
\aA_{\omega}(X/Y)= \langle \oO_{\widehat{C}}^{\aA, \perp}, 
\oO_{\widehat{C}}[-1] \rangle
\end{align*}
i.e. $\aA_{\omega}(X/Y)$ is 
the tilting with respect to the torsion pair (\ref{t:pair}).  
\end{lem} 
\begin{proof}
Since both sides are the hearts of bounded t-structures, 
it is enough to show that the LHS is contained in the RHS. 
This is equivalent to the following 
inclusions:
\begin{align}\label{inc:11}
\cC_{X/Y}^{0} &\subset \langle \oO_{\widehat{C}}^{\aA, \perp}, 
\oO_{\widehat{C}}[-1] \rangle \\
\label{inc:12}
\dL f^{\ast} \aA_{\omega} &\subset \langle \oO_{\widehat{C}}^{\aA, \perp}, 
\oO_{\widehat{C}}[-1] \rangle.
\end{align}
We first show the inclusion (\ref{inc:11}). 
By the construction of $\cC_{X/Y}^{0}$ in (\ref{def:CXY}), 
it is enough to show 
\begin{align}\label{OcOc}
\oO_{\widehat{C}}^{\cC, \perp} \subset \oO_{\widehat{C}}^{\aA, \perp}.
\end{align}
Since $\oO_{\widehat{C}} \in \dL g^{\ast} \aA_{h^{\ast}\omega}$, 
we have the following inclusion
\begin{align*}
\langle \cC_{X/Y'}^{0}, \oO_{\widehat{C}} \rangle 
\subset \langle \cC_{X/Y'}^{0}, \dL g^{\ast} \aA_{h^{\ast}\omega} \rangle
\end{align*}
or equivalently 
$\widetilde{\cC}_{X/Y'}^{0} \subset \aA_{h^{\ast}\omega}(X/Y')$. 
The inclusion (\ref{OcOc}) obviously follows from the above 
inclusion. 

Next we show the inclusion (\ref{inc:12}).
By Lemma~\ref{lem:pstab}, we have the inclusions
\begin{align*}
\dL f^{\ast} \aA_{\omega} \subset
\dL g^{\ast} \aA_{h^{\ast}\omega} \subset \aA_{h^{\ast}\omega}(X/Y'). 
\end{align*} 
Also we have
\begin{align*}
\Hom(\oO_{\widehat{C}}, \dL f^{\ast} \aA_{\omega}) &\cong 
\Hom(\dR h_{!} \oO_C, \aA_{\omega}) \\
&\cong 0
\end{align*}
since $\dR h_{!} \oO_C=0$. 
This implies that $\dL f^{\ast} \aA_{\omega}$
is contained in $\oO_{\widehat{C}}^{\aA, \perp}$, 
proving (\ref{inc:12}). 
\end{proof}

Note that $h^{\ast}A(Y)$ is a real codimension one 
boundary of $\overline{A}(Y')$. 
We define the 
subset $A_h^{\dag}(Y) \subset 
\overline{A}^{\dag}(Y')$
by the Cartesian square
\begin{align*}
\xymatrix{ \ar@{}[rd]|{\square}
A_h^{\dag}(Y) \ar[r] \ar[d]_{g_{\ast}} & \overline{A}^{\dag}(Y') 
\ar[d]^{g_{\ast}} \\
h^{\ast} A(Y) \ar[r] & \overline{A}(Y'). 
}
\end{align*}
By Proposition~\ref{prop:liftY}, we have 
\begin{align}\label{codim:one}
\sigma_{Y'}(A_{h}^{\dag}(Y)) \subset
 \overline{U}(Y')
\end{align}
and it 
is a real codimension one boundary of $\overline{U}(Y')$. 
The following proposition shows the desired
property of $\overline{U}(Y) \cap \overline{U}(Y')$. 
\begin{prop}\label{prop:main1}
We have 
\begin{align*}
\sigma_{Y'}(A_h^{\dag}(Y)) \subset \overline{U}(Y). 
\end{align*}
\end{prop}
\begin{proof}
It is enough to show the claim for rational points in $A_h^{\dag}(Y)$. 
Let us take a rational point in $A_h^{\dag}(Y)$
\begin{align*}
g^{\ast}h^{\ast}\omega + D =f^{\ast}\omega +D \in A_h^{\dag}(Y)
\end{align*}
for $\omega \in A(Y)$ and $D\in C_{g, \omega^2}(X)$. 
By (\ref{codim:one}), we have the point
\begin{align}\label{foDi}
\sigma_{f^{\ast}\omega +D} \in \overline{U}(Y'). 
\end{align} 
On the other hand, if we take a rational number
$0<t \ll 1$, which is 
sufficiently small depending on $\omega$ and $D$, 
we have 
\begin{align}\label{faoD}
f^{\ast}\omega + D + t\widehat{C} \in A^{\dag}(Y)
\end{align}
by the description of $C_{f, \omega^2}(X)$ 
in the proof of Lemma~\ref{lem:Cf}. 
Hence we have the point
\begin{align*}
\sigma_{f^{\ast}\omega + D + t\widehat{C}} \in U(Y). 
\end{align*}
It is enough to show that 
\begin{align}\label{fol:lim}
\lim_{t\to +0} \sigma_{f^{\ast}\omega + D + t\widehat{C}}
= \sigma_{f^{\ast}\omega +D}. 
\end{align}
The relation (\ref{fol:lim})
obviously follows at the level of central charges. 
Also the hearts of
bounded t-structures associated to 
(\ref{foDi}), (\ref{faoD}) are
\begin{align*}
\aA_{h^{\ast}\omega}(X/Y'), \quad 
\aA_{\omega}(X/Y) 
\end{align*}
respectively. By Lemma~\ref{final:tilting}, these 
t-structures are related by a tilting. Moreover 
the heart $\aA_{\omega}(X/Y)$ is independent of $t$. 
Therefore we can apply 
Lemma~\ref{lem:limit} below, and conclude that
the relation (\ref{fol:lim}) holds. 
\end{proof}
We have used the following lemma, which is proved in~\cite{Tcurve1}. 
\begin{lem}{\bf (\cite[Lemma~7.1]{Tcurve1}) }
\label{lem:limit}
Let $\aA, \aA'$ be the hearts of bounded t-structures on $\dD$, 
which are related by a tilting. Let
\begin{align*}
[0, 1) \ni t \mapsto Z_t \in N(X)_{\mathbb{C}}^{\vee}
\end{align*}
be a continuous map such that $\sigma_t =(Z_t, \aA)$ for
any rational number
$0< t <1$ and $\sigma_0=(Z_0, \aA')$
determine points in $\Stab(X)$. Then we 
have $\lim_{t\to +0} \sigma_t =\sigma_0$. 
\end{lem}

\subsection{Moduli spaces}
Let $\mM$ be the algebraic space which 
parameterizes objects $E\in D^b \Coh(X)$
satisfying
\begin{align*}
\Ext^{<0}(E, E)=0, \ 
\Hom(E, E)=\mathbb{C}
\end{align*}
constructed by Inaba~\cite{Inaba}. 
For $\sigma=(Z, \aA) \in \Stab(X)_{\mathbb{R}}$, let
\begin{align*}
\mM^{\sigma}([\oO_x]) \subset \mM
\end{align*}
be the subspace which parameterizes
$Z$-stable objects $E\in \aA$ with $\ch(E)=\ch(\oO_x)$
for $x\in X$. 
Note that, a priori, $\mM^{\sigma}([\oO_x])$ is just an abstract 
subfunctor of $\mM$ from
the category of $\mathbb{C}$-schemes 
to the category of sets.   
The subspace $\mM^{\sigma}([\oO_x])$ is shown to be
an algebraic subspace if 
the openness of $\sigma$-stable objects is proved. (See~\cite{Tst3}
for the arguments when $X$ is a K3 surface or an abelian surface.)
The following proposition completes the proof of Theorem~\ref{thm:intro}:
\begin{prop}\label{prop:main2}
For $\sigma \in U(Y)$, the 
space $\mM^{\sigma}([\oO_x])$ is 
an open algebraic subspace of $\mM$, and isomorphic 
to $Y$. 
\end{prop} 
\begin{proof}
By deforming $\sigma \in U(Y)$,
we may assume that $\sigma$ is written as
\begin{align*}
(Z_{f^{\ast}\omega +D}, \aA_{\omega}^{\dag}(X/Y))
\end{align*}
for some rational point 
$f^{\ast}\omega +D \in A^{\dag}(Y)$.
In order to reduce the notation, we 
write $Z=Z_{f^{\ast}\omega +D}$.  
Let us take an object $E \in \aA_{\omega}^{\dag}(X/Y)$, 
giving a $\mathbb{C}$-valued point 
of $\mM^{\sigma}([\oO_x])$. It fits 
into an exact sequence in $\aA_{\omega}^{\dag}(X/Y)$
\begin{align*}
0 \to F \to E \to \dL f^{\ast} M \to 0
\end{align*}
for some $F\in \cC_{X/Y}^{0}$ and $M \in \aA_{\omega}$.  
If $F\neq 0$, we have
$\Imm Z(F)>0$, $\Imm Z(\dL f^{\ast}M) \ge 0$, 
hence $\Imm Z(E)>0$. 
This contradicts to $\ch(E)=\ch(\oO_x)$, 
hence $F=0$
and $E \cong \dL f^{\ast} M$. 
Since $M \in \aA_{\omega}$ satisfies 
$\ch(M)=\ch(\oO_y)$, 
Lemma~\ref{Ox:stable}
implies that $M\cong \oO_y$ for some $y\in Y$,
i.e. $E \cong \dL f^{\ast} \oO_y$. 
Conversely, let us consider the object
$\dL f^{\ast}\oO_y \in \aA_{\omega}(X/Y)$. 
Since $\oO_y \in \aA_{\omega}$ is $Z_{\omega, D}$-stable 
by Lemma~\ref{Ox:stable}, Lemma~\ref{rel:pull} implies that 
$\dL f^{\ast} \oO_y \in \aA_{\omega}(X/Y)$
is also $Z$-stable. 

The above argument shows that the morphism 
\begin{align}\label{mor:YM}
Y \to \mM
\end{align}
sending $y$ to $\dL f^{\ast}\oO_y$ induces a bijection 
between closed points of $Y$ and those of $\mM^{\sigma}([\oO_x])$. 
Also since the functor 
\begin{align*}
\dL f^{\ast} \colon D^b \Coh(Y) \to D^b \Coh(X)
\end{align*}
is fully faithful, the 
morphism (\ref{mor:YM})
is bijective on the tangent spaces. 
Therefore it is enough to show that 
$\mM^{\sigma}([\oO_x])$ is open in $\mM$, which 
follows if we show that the objects of the 
form $\dL f^{\ast} \oO_y$ are closed under deformations. 

Note that an object $E\in D^b \Coh(X)$ is written as 
$\dL f^{\ast} M$ for some $M \in D^b \Coh(Y)$
if and only if 
\begin{align}\label{DXYE}
\Hom(\dD_{X/Y}, E)=0
\end{align}
 where
$\dD_{X/Y}$ is defined in Subsection~\ref{subsec:perverse}. 
Since $\dD_{X/Y}$ is the smallest triangulated 
subcategory which contains some finite 
number of objects in $\dD_{X/Y}$,
the condition (\ref{DXYE}) is an open 
condition by upper semicontinuity.
Hence if $E$ is a small 
deformation of $\dL f^{\ast} \oO_y$, it is 
of the form $\dL f^{\ast}M$ for some 
$M\in D^b \Coh(Y)$. Then $M$ is a small 
deformation of $\oO_y$, hence 
$M\cong \oO_{y'}$ 
for some $y'\in Y$. The statement is now proved. 
\end{proof}

\section{Some technical results}\label{sec:tech}
In this section, we give proofs of Proposition~\ref{prop:lift}
and Proposition~\ref{prop:liftY}. 
\subsection{Proof of Proposition~\ref{prop:lift}}\label{pf:lift}
\begin{proof}
We divide the proof into four steps. 
\begin{sstep}\label{step:pro}
Continuity of $\sigma$ at rational points. 
\end{sstep}
Let us take a rational point
$\omega \in \overline{A}(X)$. 
By Theorem~\ref{thm:Brmain} and Proposition~\ref{prop:stab}, 
there are open neighborhoods
\begin{align}\label{neighbor}
\omega \in U_{\omega} \subset \mathrm{NS}(X)_{\mathbb{R}}, \quad 
\sigma_{\omega} \in 
\uU_{\omega} \subset \Stab(X)_{\mathbb{R}}
\end{align}
such that $\Pi_{\mathbb{R}}$ restricts to a homeomorphism 
between $\uU_{\omega}$ and $U_{\omega}$. 
We claim that, 
after shrinking (\ref{neighbor}) if necessary, 
we have 
\begin{align}\label{sig:lim}
\sigma_{\omega'} \in \uU_{\omega}, 
\mbox{ 
for any rational }
\omega' \in U_{\omega} \cap \overline{A}(X).
\end{align}
To prove this, we may assume that $\omega'$
lies in the interior 
$A(X) \subset \overline{A}(X)$
since $\Stab(X)_{\mathbb{R}}$ is Hausdorff.
Let us take a stability condition 
$\widetilde{\sigma}_{\omega'} \in \uU_{\omega}$
satisfying $\Pi_{\mathbb{R}}(\widetilde{\sigma}_{\omega'})=\omega'$. 
By~\cite[Proposition~3.14]{Todext}, 
after shrinking (\ref{neighbor}) if necessary, 
any object $\oO_x$ for $x\in X$ is 
 $\widetilde{\sigma}_{\omega'}$-stable
of phase one. 
Then Lemma~\ref{lem:Brlem} below
shows that 
\begin{align*}
\widetilde{\sigma}_{\omega'}= 
\sigma_{\omega'}. 
\end{align*}
Therefore the condition (\ref{sig:lim}) holds.

\begin{sstep}\label{step:lim}
Partial extension of $\sigma$ to irrational points. 
\end{sstep}
By the property (\ref{sig:lim}) and Theorem~\ref{thm:Brmain}, 
there is an open subset $U \subset \overline{A}(X)$, 
which contains all the rational points, such that 
the construction in Proposition~\ref{prop:stab}
extends to a continuous map 
\begin{align*}
\sigma_U \colon U \to \Stab(X)_{\mathbb{R}}. 
\end{align*}
It is enough to show that $\sigma_U$ extends to the whole 
$\overline{A}(X)$. 
We first show that $\sigma_U$ extends to 
$U \cup A(X)$. 
Let us take an irrational point
$\omega \in A(X)$, 
and rational points
$\omega_j \in A(X)$ for $j\ge 1$
which converge to $\omega$.
By Proposition~\ref{prop:stab},
there is a constant $K_j>0$ such that 
\begin{align*}
\frac{ \lVert E \rVert}{\lvert Z_{\omega_j}(E) \rvert}
<K_{j}
\end{align*}  
for any non-zero 
$\sigma_{\omega_j}$-semistable object $E$. 
By the evaluation of $K_j$ in the proof of~\cite[Proposition~3.13]{Todext}, 
one can easily check that the $K_j$ is
taken to be independent of $j$. 
Indeed 
we can take $K_j$ so that 
\begin{align}\label{Kc}
K_j^2 < c_0 + c_1 \cdot C_{\omega_j} + 
c_2 \cdot l_j + c_3 \cdot m_j
\end{align}
where $c_i$ are positive constants which 
are independent of $j$, and 
$C_{\omega_j}$, $l_j$ and $m_j$ 
are given by
\begin{align*}
&C_{\omega_j} \cneq \mathrm{sup} \left\{
- \frac{D^2 \cdot \omega_j^2}{(D\cdot \omega_j)^2} : 
D \mbox{ is an effective divisor on } X, D^2 \le 0
\right\} \\
& l_j \cneq \sup \left\{ \frac{u}{(u+ \omega_j^2/2)^2} : u>0 \right\} \\
&m_j \cneq \sup \left\{
 \frac{y^2}{ (-y + \omega_j^2/2)^2 + x^2} : 
(x, y) \in \mathbb{R}^2, x^2 \ge 2\omega_j^2 y \right\}. 
\end{align*}
Here $C_{\omega_j}$ is obtained in~\cite[Lemma~3.20]{Todext}
and~\cite[Corollary~7.3.3]{BMT}, 
$l_j$ appears in the proof of~\cite[Proposition~3.19]{Todext}
and $m_j$ appears in the proof of~\cite[Proposition~3.19]{Todext}.
The argument of~\cite[Proposition~3.19]{Todext}
implies the inequality of the form (\ref{Kc}). 
By the openness of $A(X)$, the values 
$C_{\omega_j}$, $l_j$ and $m_j$
are 
are bounded above
by a positive constant which is  
independent of $j$. 

The fact that $K_j$ is bounded above 
easily implies that 
\begin{align*}
\lim_{j \to \infty} \left\{ 
\frac{\lvert Z_{\omega_j}(E) - Z_{\omega}(E)\rvert}{\lvert Z_{\omega_j}(E) 
\rvert}
: E \mbox{ is } \sigma_{\omega_j} \mbox{-semistable } \right\} =0. 
\end{align*}
Therefore, by~\cite[Theorem~7.1]{Brs1}, 
there is $\sigma_{\omega} \in \Stab(X)_{\mathbb{R}}$
satisfying
\begin{align}\label{limjs}
\lim_{j \to \infty} \sigma_{\omega_j} =
\sigma_{\omega}. 
\end{align}

\begin{sstep}\label{sstep:lift}
Well-definedness of $\sigma_{\omega}$. 
\end{sstep}
We need to show that $\sigma_{\omega}$ in (\ref{limjs})
is independent of $\omega_j$. 
In order to show this, 
we claim that $\oO_x$ is $\sigma_{\omega}$-stable
for any $x\in X$. 
Suppose that $\oO_x$ is not $\sigma_{\omega}$-stable. 
Since $\omega_j$ is rational, 
$\oO_x \in \aA_{\omega_j}$ is $\sigma_{\omega_j}$-stable
by Lemma~\ref{Ox:stable}, hence
 $\oO_x$ is $\sigma_{\omega}$-semistable.
This implies that there is a
non-trivial $\sigma_{\omega}$-stable factor $A$ of $\oO_x$, 
and $\omega$ is a solution of $\omega \cdot c_1(A)=0$. 

On the other hand, 
let us take a sufficiently small open neighborhood 
$\sigma_{\omega} \in \uU_{\omega}$. 
Since $\sigma_{\omega}$ satisfies the support property, 
there is a wall and chamber structure on $\uU_{\omega}$
with finite number of codimension one walls 
such that 
the set of semistable objects $E$ with $\ch(E)=\ch(\oO_x)$ is constant
at a chamber but jumps at a wall.
By the argument as above, $\sigma_{\omega}$ lies 
at the wall of the form $\Pi_{\mathbb{R}}(\ast) \cdot c_1(A)=0$. 
Since the image of this wall
under $\Pi_{\mathbb{R}}$ contains dense rational points, we 
can deform $\sigma_{\omega}$ to
$\widetilde{\sigma}_{\omega''}$
on the wall such that its image 
under $\Pi_{\mathbb{R}}$ is 
 a rational point $\omega'' \in A(X)$. 
Since $\widetilde{\sigma}_{\omega''}$
lies on the wall, it is a limit of 
stability conditions of the form $\sigma_{\omega_{j}''}$
for $j\ge 1$ 
with $\omega_{j}''$ rational and 
$\omega_{j}'' \to \omega''$.  However, by 
the property (\ref{sig:lim}), the stability condition 
$\sigma_{\omega''}$ is also the 
limit of $\sigma_{\omega_{j}''}$. 
Therefore $\widetilde{\sigma}_{\omega''}=\sigma_{\omega''}$, 
which is a contradiction
since $\oO_x$ is not $\widetilde{\sigma}_{\omega''}$-stable 
but $\sigma_{\omega''}$-stable. 
Therefore $\oO_x$ is $\sigma_{\omega}$-stable, 

Since $\oO_x$ is $\sigma_{\omega}$-stable, 
if we take open subsets as in (\ref{neighbor}) for
an irrational $\omega$, then 
the same argument as in Step~\ref{step:pro} shows that 
they satisfy the condition (\ref{sig:lim}). 
This immediately implies that $\sigma_{\omega}$
is independent of the choice of $\omega_j$. 
Hence $\sigma_U$ extends to the continuous 
map from $U \cup A(X)$, by sending 
$\omega$ to $\sigma_{\omega}$. 

\begin{sstep}\label{sstep:final}
Extension of $\sigma$ to all the irrational 
points. 
\end{sstep}
The final step is to extend the map from 
$U \cup A(X)$ to the map from $\overline{A}(X)$. 
Let us take an irrational point 
$\omega \in \overline{A}(X) \setminus A(X)$, 
and rational points $\omega_j \in \overline{A}(X) \setminus A(X)$
for $j\ge 1$
which converge to $\omega$. 
By the same argument as in Step~\ref{step:lim}, 
the limit $\sigma_{\omega}$ of $\sigma_{\omega_j}$ exists. 
Note that $A(X)$ is continuously embedded into $\Stab(X)_{\mathbb{R}}$
by the previous step, 
which gives a section of $\Pi_{\mathbb{R}}$ over $A(X)$. 
Since $\sigma_{\omega_j}, \sigma_{\omega}$ lie at 
its boundary, $\sigma_{\omega}$ is uniquely determined by $\omega$
if it exists, and independent of the choice of $\omega_j$. 
Now the 
assignment $\omega \mapsto \sigma_{\omega}$
gives the desired continuous map (\ref{map:lift}). 
\end{proof}

We have used the following lemma, 
which is essentially proved in~\cite{Brs2}. 
\begin{lem}\label{lem:Brlem}
Let $\aA \subset D^b \Coh(X)$ be the heart of a
bounded t-structure and $\omega \in \mathrm{NS}(X)_{\mathbb{R}}$
is ample. Suppose that the following conditions hold: 
\begin{itemize}
\item The pair $(Z_{\omega}, \aA)$ is a stability condition on $D^b \Coh(X)$. 
\item For any $x\in X$, we have $\oO_x \in \aA$, and it is $Z_{\omega}$-stable. \end{itemize}
Then we have $\aA=\aA_{\omega}$. 
\end{lem}
\begin{proof}
The result is essentially proved in~\cite[Proposition~10.3, Step~2]{Brs2}, 
using~\cite[Lemma~10.1]{Brs2}. 
Although these results in~\cite{Brs2} are stated 
for K3 surfaces or abelian surfaces, one can 
see that the arguments work for arbitrary projective surfaces. 
\end{proof}

\subsection{Proof of Proposition~\ref{prop:liftY}}\label{subsec:techY}
\begin{proof}
The proof is similar to that of Proposition~\ref{prop:lift}, 
but we need to take more care because we are 
no longer able to use Lemma~\ref{lem:Brlem}. 

\begin{step}
Continuity at rational points: 
reduction to the equality of stability conditions
given as (\ref{sig'=sig}). 
\end{step}
Let us take a rational point
$f^{\ast}\omega +D \in \overline{A}^{\dag}(Y)$. 
By Theorem~\ref{thm:Brmain} and Proposition~\ref{prop:stab}, 
there are open neighborhoods
\begin{align}\label{neighbor1}
&f^{\ast}\omega +D \in U_{\omega, D} \subset \mathrm{NS}(X)_{\mathbb{R}} \\
\label{neighbor2}
&\sigma_{f^{\ast}\omega +D} \in 
\uU_{\omega, D} \subset \Stab(X)_{\mathbb{R}}
\end{align}
such that $\Pi_{\mathbb{R}}$ restricts to a homeomorphism 
between $\uU_{\omega, D}$ and $U_{\omega, D}$. 
We claim that, 
after shrinking (\ref{neighbor1}), (\ref{neighbor2})
 if necessary, 
we have 
\begin{align}\label{sig:lim2}
\sigma_{f^{\ast}\omega'+D'} \in \uU_{\omega, D}, 
\mbox{ 
for any rational }
f^{\ast}\omega' +D' \in U_{\omega, D} \cap \overline{A}^{\dag}(Y).
\end{align}
Let us take
 $\widetilde{\sigma}_{f^{\ast}\omega' +D'} \in \uU_{\omega, D}$
whose image under $\Pi_{\mathbb{R}}$ is $f^{\ast}\omega' +D'$. 
It is enough to show
\begin{align}\label{sig'=sig}
\widetilde{\sigma}_{f^{\ast}\omega' +D'}
= \sigma_{f^{\ast}\omega' +D'}. 
\end{align}

\begin{step}\label{step:slicing}
A preparation of slicing to prove (\ref{sig'=sig}). 
\end{step}
Below, we assume that the reader is familiar with
the notion of slicings in the original paper~\cite{Brs1}, 
and the related notations. 
The stability conditions $\sigma_{f^{\ast}\omega +D}$, 
$\widetilde{\sigma}_{f^{\ast}\omega' +D'}$ are written 
as pairs
\begin{align*}
&\sigma_{f^{\ast}\omega +D}=(Z_{f^{\ast}\omega +D}, \pP) \\ 
&\widetilde{\sigma}_{f^{\ast}\omega' +D'}=(Z_{f^{\ast}\omega' +D'}, \pP')
\end{align*}
for slicings $\pP= \{ \pP(\phi)\}_{\phi \in \mathbb{R}}$, 
$\pP'=\{\pP'(\phi)\}_{\phi \in \mathbb{R}}$. 
We also consider stability conditions on $D^b \Coh(Y)$
\begin{align*}
&\sigma_{\omega, D}=(Z_{\omega, D}, \aA_{\omega}) \\ 
&\sigma_{\omega', D'}=(Z_{\omega', D'}, \aA_{\omega'})
\end{align*}
considered in Lemma~\ref{lem:prepare}. 
We denote by 
\begin{align*}
\{\qQ(\phi)\}_{\phi \in \mathbb{R}}, \ 
\{ \qQ'(\phi)\}_{\phi \in \mathbb{R}}
\end{align*} the slicings 
corresponding to $\sigma_{\omega, D}$, $\sigma_{\omega', D'}$
respectively. 

By shrinking (\ref{neighbor1}), (\ref{neighbor2}) if necessary, 
there is $0<\epsilon <1/8$ so that 
\begin{align*}
d(\pP, \pP')<\epsilon
\end{align*} where $d(\ast, \ast)$ is 
given in~\cite[Section~6]{Brs1}. 
Then the category
$\pP'(\phi)$ is the category of 
$Z_{f^{\ast}\omega'+D'}$-semistable object
in the quasi-abelian category
(cf.~the proof of~\cite[Theorem~7.1]{Brs1})
\begin{align}\label{qab}
\pP((\phi-\epsilon, \phi +\epsilon)). 
\end{align}
Also, similarly to the proof of Proposition~\ref{prop:lift}, 
Step~\ref{step:pro}, 
we see that $\sigma_{\omega', D'}$ is contained in
an open neighborhood of $\sigma_{\omega, D}$. 
Consequently, we may assume that $d(\qQ, \qQ')<\epsilon$,  and 
$\qQ'(\phi)$ is the category of $Z_{\omega', D'}$-semistable 
objects in the quasi-abelian category
\begin{align*}
\qQ((\phi-\epsilon, \phi +\epsilon)). 
\end{align*}

\begin{step}\label{step:AP}
Reduction of (\ref{sig'=sig}) to 
some statements on $\qQ((\phi-\epsilon, \phi+\epsilon))$, 
given as 
(\ref{inc:ob}) and (\ref{Fi:inc}). 
\end{step}
The relation (\ref{sig'=sig}) follows 
if we show 
\begin{align}\label{AsubP}
\aA_{f^{\ast}\omega'}^{\dag} \subset \pP'((0, 1])
\end{align} 
since both sides are hearts of bounded t-structures on $D^b \Coh(X)$. 
The inclusion (\ref{AsubP}) is equivalent to 
\begin{align}
\label{inc1}
\cC_{X/Y}^{0} \subset \pP'((0, 1]) \\
\label{inc2}
\dL f^{\ast} \aA_{\omega'} \subset \pP'((0, 1]). 
\end{align}
We first prove (\ref{inc1}). Since
$\cC_{X/Y}^{0}$ is the extension closure of a
finite number of objects, 
and $\Imm Z_{D}(\cC_{X/Y}^{0} \setminus \{0\})>0$, we have 
\begin{align}\label{theta:C}
\cC_{X/Y}^{0} \subset \pP([\theta, \theta'])
\end{align}
for some $\theta, \theta' \in (0, \pi)$. 
By shrinking (\ref{neighbor1}), (\ref{neighbor2})
if necessary, 
we may assume that 
\begin{align}\label{cond:theta}
[\theta-\epsilon, \theta' +\epsilon] \subset (0, \pi). 
\end{align}
Then the condition (\ref{inc1}) follows
since $d(\pP, \pP') < \epsilon$. 

The inclusion (\ref{inc2}) follows if 
we show $\dL f^{\ast} \qQ'(\phi) \subset \pP'(\phi)$
for any $0<\phi \le 1$. 
By the argument in Step~\ref{step:slicing}, it is enough to show that
\begin{align}\label{inc:ob}
\dL f^{\ast} \qQ((\phi-\epsilon, \phi +\epsilon)) \subset 
\pP((\phi-\epsilon, \phi +\epsilon))
\end{align}
and for any $M \in \qQ((\phi-\epsilon, \phi+\epsilon))$
and an exact sequence in $\pP((\phi-\epsilon, \phi+\epsilon))$
\begin{align}\label{FMF}
0 \to F_1 \to \dL f^{\ast}M \to F_2 \to 0 
\end{align}
we have 
\begin{align}\label{Fi:inc}
F_i \in \dL f^{\ast}\qQ((\phi-\epsilon, \phi+\epsilon)).
\end{align}

\begin{step}\label{step:LQP}
Proof of 
(\ref{inc:ob}) and (\ref{Fi:inc}).
\end{step}
By Lemma~\ref{rel:pull}, we have 
\begin{align}\label{fQP}
\dL f^{\ast} \qQ(\psi) \subset \pP(\psi)
\end{align}
for any $\psi \in \mathbb{R}$. Then 
the inclusion (\ref{inc:ob}) is obvious from 
(\ref{fQP}). 
Suppose that there is an exact sequence (\ref{FMF}) 
in $\pP((\phi-\epsilon, \phi+\epsilon))$. 
If $\phi \in (\epsilon, 1-\epsilon)$, we have 
$\pP((\phi-\epsilon, \phi+\epsilon)) \subset \aA_{\omega}(X/Y)$. 
By Lemma~\ref{closed:under},
it follows 
that
$F_i \in \dL f^{\ast}\aA_{\omega}$
for $i=1, 2$. 
Together with 
$F_i \in \pP((\phi-\epsilon, \phi+\epsilon))$
and the condition (\ref{fQP}), we conclude that
(\ref{Fi:inc}) holds. 

If $\phi \notin (\epsilon, 1-\epsilon)$, 
we have either $\phi \in (0, \epsilon)$
or $\phi \in (1-\epsilon, 1]$. 
These cases are treated similarly, so 
we assume $\phi \in (1-\epsilon, 1]$ for simplicity. 
By setting $\aA=\aA_{\omega}(X/Y)$, we have 
the exact sequence in $\aA$
\begin{align}\notag
0 &\to \hH_{\aA}^{-1}(F_1) \to \hH_{\aA}^{-1}(\dL f^{\ast}M) \to 
\hH_{\aA}^{-1}(F_2) \\
\label{long:ex}
&\to 
\hH_{\aA}^{0}(F_1) \to \hH_{\aA}^{0}(\dL f^{\ast}M) \to 
\hH_{\aA}^{0}(F_2) \to 0. 
\end{align}
Since $\dL f^{\ast}\aA_{\omega} \subset \aA_{\omega}(X/Y)$, we have 
\begin{align*}
\hH_{\aA}^{i} (\dL f^{\ast}M) \cong \dL f^{\ast} \hH_{\aA_{\omega}}^{i}(M)
\end{align*}
for all $i$. 
By Lemma~\ref{closed:under} and the exact
sequence (\ref{long:ex}), we have
\begin{align*}
\hH_{\aA}^{-1}(F_1), \hH_{\aA}^{0}(F_2) \in \dL f^{\ast}\aA_{\omega}.
\end{align*}
By the condition (\ref{cond:theta}), we have
\begin{align*}
\Hom(\cC_{X/Y}^{0}, \hH_{\aA}^{-1}(F_2))=0
\end{align*}
since $\hH_{\aA}^{-1}(F_2) \in \pP((0, \epsilon))$. 
This implies that $\hH_{\aA}^{-1}(F_2) \in \dL f^{\ast}\aA_{\omega}$, 
hence $\hH_{\aA}^{0}(F_1) \in \dL f^{\ast}\aA_{\omega}$ also
holds by Lemma~\ref{closed:under}
and the exact sequence (\ref{long:ex}).

We have shown that $\hH_{\aA}^{j}(F_i)$ is an object
in $\dL f^{\ast} \aA_{\omega}$ for all $i$ and $j$. 
Since the functor
\begin{align*}
\dL f^{\ast} \colon D^b \Coh(Y) \to D^b \Coh(X)
\end{align*}
is fully faithful, it follows that 
$F_i \in \dL f^{\ast} D^b \Coh(Y)$
for $i=1, 2$. 
Combined with (\ref{fQP}), we conclude that (\ref{Fi:inc}) holds. 

\begin{step}\label{step:lim2}
Partial extension of $\sigma$ to irrational points. 
\end{step}
Now we have proved (\ref{sig:lim2}). 
By the property (\ref{sig:lim2}), there is an open 
subset $U_Y \subset \overline{A}^{\dag}(Y)$, 
which contains all the rational points, such that 
the construction in Proposition~\ref{prop:Adag}
extends to a continuous map
\begin{align*}
\sigma_{U, Y} \colon U_Y \to \Stab(X)_{\mathbb{R}}. 
\end{align*}
We next show that $\sigma_{U, Y}$ extends to the 
$U_Y \cup A^{\dag}(Y)$. 
For an irrational point
$f^{\ast}\omega +D \in A^{\dag}(Y)$, 
let us take rational points 
$f^{\ast}\omega_j +D_j \in A^{\dag}(Y)$ for $j\ge 1$
which converge to $f^{\ast}\omega +D$. 
By Lemma~\ref{lem:Adag}, there is a constant $K_j>0$ so that
\begin{align*}
\frac{\lVert E \rVert}{\lvert Z_{f^{\ast}\omega_j +D_j}(E) \rvert} <K_j
\end{align*}
for any non-zero $\sigma_{f^{\ast}\omega +D_j}$-semistable 
object $E$. By the evaluation of $K_j$ in the proof of
Proposition~\ref{prop:Adag}, and the argument in the proof of 
Proposition~\ref{prop:lift}, Step~\ref{step:lim}, 
it is easy to see that the constant $K_j$ 
is taken to be independent of $j$. 
Therefore, as in the proof of Proposition~\ref{prop:lift}, Step~\ref{step:lim}, 
the limit exists
\begin{align}\label{lim:j}
\sigma_{f^{\ast}\omega +D} \cneq 
\lim_{j\to \infty} \sigma_{f^{\ast}\omega_j +D_j}. 
\end{align}

\begin{step}
Well-definedness of $\sigma_{f^{\ast}\omega +D}$. 
\end{step}
We claim that the limit (\ref{lim:j})
does not depend on a choice of $f^{\ast}\omega_j +D_j$.
Indeed if we take another
rational points
 $f^{\ast}\omega_j' +D_j' \in A^{\dag}(Y)$
which converge to $f^{\ast}\omega +D$,
then Theorem~\ref{thm:Brmain} implies the 
existence of
\begin{align*}
\widetilde{\sigma}_{f^{\ast}\omega_j'+ D_j'} \in \Stab(X)_{\mathbb{R}}
\end{align*}
for $j \gg 0$
whose image under $\Pi_{\mathbb{R}}$ is $f^{\ast}\omega_j'+D_j'$
and converge to $\sigma_{f^{\ast}\omega +D}$. 
Let us write
\begin{align*}
\sigma_{f^{\ast}\omega_j +D_j}=
(Z_{f^{\ast}\omega_j +D_j}, \pP_j) \\
\widetilde{\sigma}_{f^{\ast}\omega_j' +D_j'}
=(Z_{f^{\ast}\omega_j' +D_j'}, \pP_j')
\end{align*}
for slicings $\pP_j=\{\pP_j(\phi)\}_{\phi \in \mathbb{R}}$
and $\pP_j'=\{ \pP_j'(\phi)\}_{\phi \in \mathbb{R}}$. 
Then $d(\pP_j, \pP_j')$ goes to zero for $j\to \infty$. 
Also we can take $\theta, \theta' \in (0, \pi)$, 
which does not depend on $j$, so that
\begin{align*}
\cC_{X/Y}^{0} \subset \pP_j([\theta, \theta'])
\end{align*}
for all $j\gg 0$. 
If we take $0<\epsilon <1/8$ satisfying (\ref{cond:theta}), 
we have 
\begin{align}\label{inc:fin1}
\cC_{X/Y}^{0} \subset \pP_j'((0, 1])
\end{align}
for all $j\gg 0$ satisfying $d(\pP_j, \pP_j')<\epsilon$. 
Also, 
by the same argument of 
Proposition~\ref{prop:lift}, Step~\ref{sstep:lift},
one sees that the stability conditions 
$\sigma_{\omega_j, D_j}$ and $\sigma_{\omega_j', D_j'}$
converge to the same point in $\Stab(Y)_{\mathbb{R}}$. 
Using this fact
instead of
 the two sentences after (\ref{qab}), 
the same argument proving (\ref{inc2}) 
shows the inclusion
\begin{align}\label{inc:fin2}
\dL f^{\ast} \aA_{\omega_j'} \subset \pP_j'((0, 1]). 
\end{align}
The inclusions (\ref{inc:fin1}), (\ref{inc:fin2}) imply 
$\pP_j'((0, 1])=\aA_{f^{\ast}\omega_j'}^{\dag}$
for $j\gg 0$, 
which implies 
\begin{align*}
\widetilde{\sigma}_{f^{\ast}\omega_j' +D_j'}
=\sigma_{f^{\ast}\omega_j' +D_j'}, \quad j\gg 0. 
\end{align*}
Hence $\sigma_{f^{\ast}\omega +D}$ is independent of 
$f^{\ast}\omega_j'+D_j'$, and the assignment
$f^{\ast}\omega +D \mapsto \sigma_{f^{\ast}\omega +D}$ 
gives a continuous map from $U_Y \cup A^{\dag}(Y)$. 

\begin{step}
Extension of $\sigma$ to all the irrational points. 
\end{step}
We finally extend the map from 
$U_Y \cup A^{\dag}(Y)$ to the map from $\overline{A}^{\dag}(Y)$.
Let us take an irrational point $f^{\ast}\omega +D \in \overline{A}^{\dag}(Y)
\setminus A^{\dag}(Y)$, and 
rational points 
$f^{\ast}\omega_j +D_j \in \overline{A}^{\dag}(Y) \setminus A^{\dag}(Y)$
for $j\ge 1$
which converge to $f^{\ast}\omega +D$. Similarly to the argument of 
Step~\ref{step:lim2},
the limit $\sigma_{f^{\ast}\omega +D}$ of 
$\sigma_{f^{\ast}\omega_j +D_j}$ exists. 
Since we have shown that $A^{\dag}(Y)$ is continuously embedded
into $\Stab(X)_{\mathbb{R}}$, the same 
argument in the proof of Proposition~\ref{pf:lift}, Step~\ref{sstep:final}
shows that $\sigma_{f^{\ast}\omega +D}$ is independent of 
$f^{\ast}\omega_j +D_j$. 
Now the assignment $f^{\ast}\omega +D \mapsto \sigma_{f^{\ast}\omega +D}$
gives the desired continuous map (\ref{cont:liftY}). 

\end{proof}

Todai Institute for Advanced Studies (TODIAS), 

Kavli Institute for the Physics and 
Mathematics of the Universe, 

University of Tokyo, 
5-1-5 Kashiwanoha, Kashiwa, 277-8583, Japan.

\textit{E-mail address}: yukinobu.toda@ipmu.jp

\end{document}